\documentclass[12pt,reqno]{amsart}
\usepackage{amssymb}
\usepackage{eucal}
\usepackage{amsmath}
\usepackage{amscd}
\usepackage[dvips]{color}
\usepackage{multicol}
\usepackage[all]{xy}           %xypic macro for latex2.09
\usepackage{graphicx}
\usepackage{color}
\usepackage{colordvi}
\usepackage{xspace}
\usepackage{tikz}

\usepackage{ifpdf}
\ifpdf
 \usepackage[colorlinks,final,backref=page,hyperindex]{hyperref}
\else
 \usepackage[colorlinks,final,backref=page,hyperindex,hypertex]{hyperref}
\fi

\begin{document}
\newcommand {\emptycomment}[1]{} %to remove paragraphs

\newcommand{\nc}{\newcommand}
\newcommand{\delete}[1]{}
\nc{\mfootnote}[1]{\footnote{#1}} % Use this to show footnotes
\nc{\todo}[1]{\tred{To do:} #1}

%\delete{
\nc{\mlabel}[1]{\label{#1}}  % Use this to suppress names
\nc{\mcite}[1]{\cite{#1}}  % Use this to suppress names
\nc{\mref}[1]{\ref{#1}}  % Use this to suppress names
\nc{\meqref}[1]{\eqref{#1}} % Use this to suppress names
\nc{\mbibitem}[1]{\bibitem{#1}} % Use this to show number
%}

\delete{
\nc{\mlabel}[1]{\label{#1}  % Use the next two lines to show names
{\hfill \hspace{1cm}{\bf{{\ }\hfill(#1)}}}}
\nc{\mcite}[1]{\cite{#1}{{\bf{{\ }(#1)}}}}  % Use this lines to show names
\nc{\mref}[1]{\ref{#1}{{\bf{{\ }(#1)}}}}  % Use this lines to show names
\nc{\meqref}[1]{\eqref{#1}{{\bf{{\ }(#1)}}}} % Use this lines to show names
\nc{\mbibitem}[1]{\bibitem[\bf #1]{#1}} % Use this to show name
}

%%%%%%%%%%%%%%%%%%%%%%%% Statements
\newtheorem{thm}{Theorem}[section]
\newtheorem{lem}[thm]{Lemma}
\newtheorem{cor}[thm]{Corollary}
\newtheorem{pro}[thm]{Proposition}
\theoremstyle{definition}
\newtheorem{defi}[thm]{Definition}
\newtheorem{ex}[thm]{Example}
\newtheorem{rmk}[thm]{Remark}
\newtheorem{pdef}[thm]{Proposition-Definition}
\newtheorem{condition}[thm]{Condition}

\renewcommand{\labelenumi}{{\rm(\alph{enumi})}}
\renewcommand{\theenumi}{\alph{enumi}}
\renewcommand{\labelenumii}{{\rm(\roman{enumii})}}
\renewcommand{\theenumii}{\roman{enumii}}

\nc{\tred}[1]{\textcolor{red}{#1}}
\nc{\tblue}[1]{\textcolor{blue}{#1}}
\nc{\tgreen}[1]{\textcolor{green}{#1}}
\nc{\tpurple}[1]{\textcolor{purple}{#1}}
\nc{\btred}[1]{\textcolor{red}{\bf #1}}
\nc{\btblue}[1]{\textcolor{blue}{\bf #1}}
\nc{\btgreen}[1]{\textcolor{green}{\bf #1}}
\nc{\btpurple}[1]{\textcolor{purple}{\bf #1}}

\nc{\cmn}[1]{\textcolor{purple}{Chengming New:#1}}
\nc{\cmj}[1]{\textcolor{red}{Chengming 7/1:#1}}
\nc{\cm}[1]{\textcolor{red}{Chengming:#1}}
\nc{\li}[1]{\textcolor{blue}{#1}}
\nc{\lir}[1]{\textcolor{blue}{Li:#1}}

%%%%%%%%%%%%%% Matrix symbols.Ju

\nc{\twovec}[2]{\left(\begin{array}{c} #1 \\ #2\end{array} \right )}
\nc{\threevec}[3]{\left(\begin{array}{c} #1 \\ #2 \\ #3 \end{array}\right )}
\nc{\twomatrix}[4]{\left(\begin{array}{cc} #1 & #2\\ #3 & #4 \end{array} \right)}
\nc{\threematrix}[9]{{\left(\begin{matrix} #1 & #2 & #3\\ #4 & #5 & #6 \\ #7 & #8 & #9 \end{matrix} \right)}}
\nc{\twodet}[4]{\left|\begin{array}{cc} #1 & #2\\ #3 & #4 \end{array} \right|}

\nc{\rk}{\mathrm{r}}
\newcommand{\g}{\mathfrak g}
\newcommand{\h}{\mathfrak h}
\newcommand{\pf}{\noindent{$Proof$.}\ }
\newcommand{\frkg}{\mathfrak g}
\newcommand{\frkh}{\mathfrak h}
\newcommand{\Id}{\rm{Id}}
\newcommand{\gl}{\mathfrak {gl}}
\newcommand{\ad}{\mathrm{ad}}
\newcommand{\add}{\frka\frkd}
\newcommand{\frka}{\mathfrak a}
\newcommand{\frkb}{\mathfrak b}
\newcommand{\frkc}{\mathfrak c}
\newcommand{\frkd}{\mathfrak d}
\newcommand {\comment}[1]{{\marginpar{*}\scriptsize\textbf{Comments:} #1}}
%%%%%%%%%%%%%%%%%%%%%%% symbols

\nc{\tforall}{\text{ for all }}

\nc{\sr}{\mathbf{r}}
\nc{\gen}{generalized\xspace}
\nc{\Iso}{\mathrm{Iso}}
\nc{\bk}{\bfk}
\nc{\ndhom}{\mathrm{NDHom}}
\nc{\ndt}{\mathrm{ND}}

\nc{\mR}{\alpha}

\nc{\tsymm}{extended symmetrizer\xspace}
\nc{\tsymms}{extended symmetrizers\xspace}

\nc{\inv}{symmetrized invariant\xspace}

%\nc{\nybe}{nonhomogeneous AYBE\xspace}
%\nc{\Nybe}{Nonhomogeneous AYBE\xspace}
\nc{\nybe}{nhacYBe\xspace}
\nc{\Nybe}{NhacYBe\xspace}

\nc{\strong}{strong\xspace}
\nc{\snybe}{\strong \nybe\xspace}

\nc{\rwith}{right twisted by\xspace}
\nc{\lwith}{left twisted by\xspace}

\nc{\svec}[2]{{\tiny\left(\begin{matrix}#1\\
#2\end{matrix}\right)\,}}  % column vector
\nc{\ssvec}[2]{{\tiny\left(\begin{matrix}#1\\
#2\end{matrix}\right)\,}} % subscript column vector

\nc{\typeI}{local cocycle $3$-Lie bialgebra\xspace}
\nc{\typeIs}{local cocycle $3$-Lie bialgebras\xspace}
\nc{\typeII}{double construction $3$-Lie bialgebra\xspace}
\nc{\typeIIs}{double construction $3$-Lie bialgebras\xspace}

\nc{\bia}{{$\mathcal{P}$-bimodule ${\bf k}$-algebra}\xspace}
\nc{\bias}{{$\mathcal{P}$-bimodule ${\bf k}$-algebras}\xspace}

\nc{\rmi}{{\mathrm{I}}}
\nc{\rmii}{{\mathrm{II}}}
\nc{\rmiii}{{\mathrm{III}}}
\nc{\pr}{{\mathrm{pr}}}
\newcommand{\huaA}{\mathcal{A}}

\nc{\OT}{constant $\theta$-}
\nc{\T}{$\theta$-}
\nc{\IT}{inverse $\theta$-}

%new new commands

\nc{\asi}{ASI\xspace}
\nc{\qadm}{$Q$-admissible\xspace}
\nc{\aybe}{AYBE\xspace}
\nc{\admset}{\{\pm x\}\cup K^{\times} x^{-1}}

\nc{\dualrep}{gives a dual representation\xspace}
\nc{\admt}{admissible to\xspace}

\nc{\opa}{\cdot_A}
\nc{\opb}{\cdot_B}

\nc{\post}{positive type\xspace}
\nc{\negt}{negative type\xspace}
\nc{\invt}{inverse type\xspace}

\nc{\pll}{\beta}
\nc{\plc}{\epsilon}

\nc{\ass}{{\mathit{Ass}}}
\nc{\lie}{{\mathit{Lie}}}
\nc{\comm}{{\mathit{Comm}}}
\nc{\dend}{{\mathit{Dend}}}
\nc{\zinb}{{\mathit{Zinb}}}
\nc{\tdend}{{\mathit{TDend}}}
\nc{\prelie}{{\mathit{preLie}}}
\nc{\postlie}{{\mathit{PostLie}}}
\nc{\quado}{{\mathit{Quad}}}
\nc{\octo}{{\mathit{Octo}}}
\nc{\ldend}{{\mathit{ldend}}}
\nc{\lquad}{{\mathit{LQuad}}}

 \nc{\adec}{\check{;}} \nc{\aop}{\alpha}
\nc{\dftimes}{\widetilde{\otimes}} \nc{\dfl}{\succ} \nc{\dfr}{\prec}
\nc{\dfc}{\circ} \nc{\dfb}{\bullet} \nc{\dft}{\star}
\nc{\dfcf}{{\mathbf k}} \nc{\apr}{\cdot} \nc{\spr}{\cdot}
\nc{\twopr}{\circ} \nc{\tspr}{\star} \nc{\sempr}{\ast }
\nc{\disp}[1]{\displaystyle{#1}}
\nc{\bin}[2]{ (_{\stackrel{\scs{#1}}{\scs{#2}}})}  %binomial coeff
\nc{\binc}[2]{ \left (\!\! \begin{array}{c} \scs{#1}\\
    \scs{#2} \end{array}\!\! \right )}  %binomial coeff
\nc{\bincc}[2]{  \left ( {\scs{#1} \atop
    \vspace{-.5cm}\scs{#2}} \right )}  %binomial coeff
\nc{\sarray}[2]{\begin{array}{c}#1 \vspace{.1cm}\\ \hline
    \vspace{-.35cm} \\ #2 \end{array}}
\nc{\bs}{\bar{S}} \nc{\dcup}{\stackrel{\bullet}{\cup}}
\nc{\dbigcup}{\stackrel{\bullet}{\bigcup}} \nc{\etree}{\big |}
\nc{\la}{\longrightarrow} \nc{\fe}{\'{e}} \nc{\rar}{\rightarrow}
\nc{\dar}{\downarrow} \nc{\dap}[1]{\downarrow
\rlap{$\scriptstyle{#1}$}} \nc{\uap}[1]{\uparrow
\rlap{$\scriptstyle{#1}$}} \nc{\defeq}{\stackrel{\rm def}{=}}
\nc{\dis}[1]{\displaystyle{#1}} \nc{\dotcup}{\,
\displaystyle{\bigcup^\bullet}\ } \nc{\sdotcup}{\tiny{
\displaystyle{\bigcup^\bullet}\ }} \nc{\hcm}{\ \hat{,}\ }
\nc{\hcirc}{\hat{\circ}} \nc{\hts}{\hat{\shpr}}
\nc{\lts}{\stackrel{\leftarrow}{\shpr}}
\nc{\rts}{\stackrel{\rightarrow}{\shpr}} \nc{\lleft}{[}
\nc{\lright}{]} \nc{\uni}[1]{\tilde{#1}} \nc{\wor}[1]{\check{#1}}
\nc{\free}[1]{\bar{#1}} \nc{\den}[1]{\check{#1}} \nc{\lrpa}{\wr}
\nc{\curlyl}{\left \{ \begin{array}{c} {} \\ {} \end{array}
    \right .  \!\!\!\!\!\!\!}
\nc{\curlyr}{ \!\!\!\!\!\!\!
    \left . \begin{array}{c} {} \\ {} \end{array}
    \right \} }
\nc{\leaf}{\ell}       % number of leafs
\nc{\longmid}{\left | \begin{array}{c} {} \\ {} \end{array}
    \right . \!\!\!\!\!\!\!}
\nc{\ot}{\otimes} \nc{\sot}{{\scriptstyle{\ot}}}
\nc{\otm}{\overline{\ot}}
\nc{\ora}[1]{\stackrel{#1}{\rar}}
\nc{\ola}[1]{\stackrel{#1}{\la}}%${\Bbb Z}$
\nc{\pltree}{\calt^\pl}
\nc{\epltree}{\calt^{\pl,\NC}}
\nc{\rbpltree}{\calt^r}
\nc{\scs}[1]{\scriptstyle{#1}} \nc{\mrm}[1]{{\rm #1}}
\nc{\dirlim}{\displaystyle{\lim_{\longrightarrow}}\,}
\nc{\invlim}{\displaystyle{\lim_{\longleftarrow}}\,}
\nc{\mvp}{\vspace{0.5cm}} \nc{\svp}{\vspace{2cm}}
\nc{\vp}{\vspace{8cm}} \nc{\proofbegin}{\noindent{\bf Proof: }}
%\nc{\proofbegin}{\begin{proof}} % AMS command
\nc{\proofend}{$\blacksquare$ \vspace{0.5cm}}
%\nc{\proofend}{\end{proof}} %AMS command
\nc{\freerbpl}{{F^{\mathrm RBPL}}}
\nc{\sha}{{\mbox{\cyr X}}}  %used to be \cyr
\nc{\ncsha}{{\mbox{\cyr X}^{\mathrm NC}}} \nc{\ncshao}{{\mbox{\cyr
X}^{\mathrm NC,\,0}}}
\nc{\shpr}{\diamond}    %Shuffle product
\nc{\shprm}{\overline{\diamond}}    %Shuffle product
\nc{\shpro}{\diamond^0}    %Shuffle product
\nc{\shprr}{\diamond^r}     %product on controlled trees
\nc{\shpra}{\overline{\diamond}^r}
\nc{\shpru}{\check{\diamond}} \nc{\catpr}{\diamond_l}
\nc{\rcatpr}{\diamond_r} \nc{\lapr}{\diamond_a}
\nc{\sqcupm}{\ot}
\nc{\lepr}{\diamond_e} \nc{\vep}{\varepsilon} \nc{\labs}{\mid\!}
\nc{\rabs}{\!\mid} \nc{\hsha}{\widehat{\sha}}
\nc{\lsha}{\stackrel{\leftarrow}{\sha}}
\nc{\rsha}{\stackrel{\rightarrow}{\sha}} \nc{\lc}{\lfloor}
\nc{\rc}{\rfloor}
\nc{\tpr}{\sqcup}
\nc{\nctpr}{\vee}
\nc{\plpr}{\star}
\nc{\rbplpr}{\bar{\plpr}}
\nc{\sqmon}[1]{\langle #1\rangle}
\nc{\forest}{\calf}
\nc{\altx}{\Lambda_X} \nc{\vecT}{\vec{T}} \nc{\onetree}{\bullet}
\nc{\Ao}{\check{A}}
\nc{\seta}{\underline{\Ao}}
\nc{\deltaa}{\overline{\delta}}
\nc{\trho}{\tilde{\rho}}

\nc{\rpr}{\circ}
%\nc{\apr}{\cdot}
\nc{\dpr}{{\tiny\diamond}}
\nc{\rprpm}{{\rpr}}

%%%%%%%%%%%%%%%%%%%%% roman fonts, in alphabetic order
\nc{\mmbox}[1]{\mbox{\ #1\ }} \nc{\ann}{\mrm{ann}}
\nc{\Aut}{\mrm{Aut}} \nc{\can}{\mrm{can}}
\nc{\twoalg}{{two-sided algebra}\xspace}
\nc{\colim}{\mrm{colim}}
\nc{\Cont}{\mrm{Cont}} \nc{\rchar}{\mrm{char}}
\nc{\cok}{\mrm{coker}} \nc{\dtf}{{R-{\rm tf}}} \nc{\dtor}{{R-{\rm
tor}}}
\renewcommand{\det}{\mrm{det}}
\nc{\depth}{{\mrm d}}
\nc{\Div}{{\mrm Div}} \nc{\End}{\mrm{End}} \nc{\Ext}{\mrm{Ext}}
\nc{\Fil}{\mrm{Fil}} \nc{\Frob}{\mrm{Frob}} \nc{\Gal}{\mrm{Gal}}
\nc{\GL}{\mrm{GL}} \nc{\Hom}{\mrm{Hom}} \nc{\hsr}{\mrm{H}}
\nc{\hpol}{\mrm{HP}} \nc{\id}{\mrm{id}} \nc{\im}{\mrm{im}}
\nc{\incl}{\mrm{incl}} \nc{\length}{\mrm{length}}
\nc{\LR}{\mrm{LR}} \nc{\mchar}{\rm char} \nc{\NC}{\mrm{NC}}
\nc{\mpart}{\mrm{part}} \nc{\pl}{\mrm{PL}}
\nc{\ql}{{\QQ_\ell}} \nc{\qp}{{\QQ_p}}
\nc{\rank}{\mrm{rank}} \nc{\rba}{\rm{RBA }} \nc{\rbas}{\rm{RBAs }}
\nc{\rbpl}{\mrm{RBPL}}
\nc{\rbw}{\rm{RBW }} \nc{\rbws}{\rm{RBWs }} \nc{\rcot}{\mrm{cot}}
\nc{\rest}{\rm{controlled}\xspace}
\nc{\rdef}{\mrm{def}} \nc{\rdiv}{{\rm div}} \nc{\rtf}{{\rm tf}}
\nc{\rtor}{{\rm tor}} \nc{\res}{\mrm{res}} \nc{\SL}{\mrm{SL}}
\nc{\Spec}{\mrm{Spec}} \nc{\tor}{\mrm{tor}} \nc{\Tr}{\mrm{Tr}}
\nc{\mtr}{\mrm{sk}}

%%%%%%%%%%%%%%%%%% bold face
\nc{\ab}{\mathbf{Ab}} \nc{\Alg}{\mathbf{Alg}}
\nc{\Algo}{\mathbf{Alg}^0} \nc{\Bax}{\mathbf{Bax}}
\nc{\Baxo}{\mathbf{Bax}^0} \nc{\RB}{\mathbf{RB}}
\nc{\RBo}{\mathbf{RB}^0} \nc{\BRB}{\mathbf{RB}}
\nc{\Dend}{\mathbf{DD}} \nc{\bfk}{{\bf k}} \nc{\bfone}{{\bf 1}}
\nc{\base}[1]{{a_{#1}}} \nc{\detail}{\marginpar{\bf More detail}
    \noindent{\bf Need more detail!}
    \svp}
\nc{\Diff}{\mathbf{Diff}} \nc{\gap}{\marginpar{\bf
Incomplete}\noindent{\bf Incomplete!!}
    \svp}
\nc{\FMod}{\mathbf{FMod}} \nc{\mset}{\mathbf{MSet}}
\nc{\rb}{\mathrm{RB}} \nc{\Int}{\mathbf{Int}}
\nc{\Mon}{\mathbf{Mon}}
%\nc{\remark}{\noindent{\bf Remark: }}
\nc{\remarks}{\noindent{\bf Remarks: }}
\nc{\OS}{\mathbf{OS}} %free operated semigroup
\nc{\Rep}{\mathbf{Rep}}
\nc{\Rings}{\mathbf{Rings}} \nc{\Sets}{\mathbf{Sets}}
\nc{\DT}{\mathbf{DT}}

%%%%%%%%%%%%%%%%%%%Bbb fonts
\nc{\BA}{{\mathbb A}} \nc{\CC}{{\mathbb C}} \nc{\DD}{{\mathbb D}}
\nc{\EE}{{\mathbb E}} \nc{\FF}{{\mathbb F}} \nc{\GG}{{\mathbb G}}
\nc{\HH}{{\mathbb H}} \nc{\LL}{{\mathbb L}} \nc{\NN}{{\mathbb N}}
\nc{\QQ}{{\mathbb Q}} \nc{\RR}{{\mathbb R}} \nc{\BS}{{\mathbb{S}}} \nc{\TT}{{\mathbb T}}
\nc{\VV}{{\mathbb V}} \nc{\ZZ}{{\mathbb Z}}

%%%%%%%%%%%%%%%%%%% cal fonts

\nc{\calao}{{\mathcal A}} \nc{\cala}{{\mathcal A}}
\nc{\calc}{{\mathcal C}} \nc{\cald}{{\mathcal D}}
\nc{\cale}{{\mathcal E}} \nc{\calf}{{\mathcal F}}
\nc{\calfr}{{{\mathcal F}^{\,r}}} \nc{\calfo}{{\mathcal F}^0}
\nc{\calfro}{{\mathcal F}^{\,r,0}} \nc{\oF}{\overline{F}}
\nc{\calg}{{\mathcal G}} \nc{\calh}{{\mathcal H}}
\nc{\cali}{{\mathcal I}} \nc{\calj}{{\mathcal J}}
\nc{\call}{{\mathcal L}} \nc{\calm}{{\mathcal M}}
\nc{\caln}{{\mathcal N}} \nc{\calo}{{\mathcal O}}
\nc{\calp}{{\mathcal P}} \nc{\calq}{{\mathcal Q}} \nc{\calr}{{\mathcal R}}
\nc{\calt}{{\mathcal T}} \nc{\caltr}{{\mathcal T}^{\,r}}
\nc{\calu}{{\mathcal U}} \nc{\calv}{{\mathcal V}}
\nc{\calw}{{\mathcal W}} \nc{\calx}{{\mathcal X}}
\nc{\CA}{\mathcal{A}}

%%%%%%%%%%%%%%%%%%  frak fonts
\nc{\fraka}{{\mathfrak a}} \nc{\frakB}{{\mathfrak B}}
\nc{\frakb}{{\mathfrak b}} \nc{\frakd}{{\mathfrak d}}
\nc{\oD}{\overline{D}}
\nc{\frakF}{{\mathfrak F}} \nc{\frakg}{{\mathfrak g}}
\nc{\frakm}{{\mathfrak m}} \nc{\frakM}{{\mathfrak M}}
\nc{\frakMo}{{\mathfrak M}^0} \nc{\frakp}{{\mathfrak p}}
\nc{\frakS}{{\mathfrak S}} \nc{\frakSo}{{\mathfrak S}^0}
\nc{\fraks}{{\mathfrak s}} \nc{\os}{\overline{\fraks}}
\nc{\frakT}{{\mathfrak T}}
\nc{\oT}{\overline{T}}
%\nc{\frakx}{{\mathfrak x}}
\nc{\frakX}{{\mathfrak X}} \nc{\frakXo}{{\mathfrak X}^0}
\nc{\frakx}{{\mathbf x}}
%\nc{\frakTxo}{{\frakTx}^0}
\nc{\frakTx}{\frakT}      %All rooted trees, correspond to \ncsha(X)
\nc{\frakTa}{\frakT^a}        % rooted trees for \ncsha(A)
\nc{\frakTxo}{\frakTx^0}   % rooted trees for \ncshao(X)
\nc{\caltao}{\calt^{a,0}}   % rooted trees for \ncshao(A)
\nc{\ox}{\overline{\frakx}} \nc{\fraky}{{\mathfrak y}}
\nc{\frakz}{{\mathfrak z}} \nc{\oX}{\overline{X}}

\font\cyr=wncyr10

\nc{\al}{\alpha}
\nc{\lam}{\lambda}
\nc{\lr}{\longrightarrow}
%defined by yi
\nc{\loc}[2]{#1_{(#2)}}
\nc{\vap}{\varepsilon} \nc{\str}{strong\xspace}

%%%%%%%%%%%%%%%%%%%
\nc{\xing}[1]{\textcolor{purple}{Xing:\, #1}}
\nc{\yi}[1]{\textcolor{cyan}{Yi: #1}}
\nc{\bai}[1]{\textcolor{purple}{Bai:#1}}
\nc{\revise}[1]{\textcolor{red}{#1}}
%%%%%%%%%%%%%%%%%%%%%%%%%%%%%%%%%%%%%%%%%%%%%%%%%%%%%%%%%%%%%%%%%%

\title[Operator forms of the nonhomogeneous associative Yang-Baxter
equation] {Operator forms of the nonhomogeneous associative classical
Yang-Baxter equation}

\author{Chengming Bai}
\address{Chern Institute of Mathematics \& LPMC, Nankai University, Tianjin 300071, China}
         \email{baicm@nankai.edu.cn}

\author{Xing Gao}
\address{School of Mathematics and Statistics,
    %Key Laboratory of Applied Mathematics and Complex Systems,
    Lanzhou University, Lanzhou, Gansu 730000, China}
         \email{gaoxing@lzu.edu.cn}

\author{Li Guo}
\address{Department of Mathematics and Computer Science, Rutgers University, Newark, NJ 07102, United States}
         \email{liguo@rutgers.edu}

\author{Yi Zhang}
\address{School of Mathematics and Statistics, Lanzhou University, Lanzhou, Gansu 730000, China}
         \email{zhangy2016@lzu.edu.cn}

\date{\today}

\begin{abstract}
This paper studies operator forms of the nonhomogeneous
associative classical Yang-Baxter equation (\nybe), extending and
generalizing such studies for the classical Yang-Baxter equation
and associative Yang-Baxter equation that can be tracked back to
the works of Semonov-Tian-Shansky and Kupershmidt on Rota-Baxter
Lie algebras and $\calo$-operators. Solutions of the
\nybe are characterized in terms of generalized $\calo$-operators, and in terms of the classical $\mathcal
O$-operators precisely when the solutions satisfy an invariant
condition.  When the invariant condition is compatible with a
Frobenius algebra, such solutions have a close relationship with
Rota-Baxter operators on the Frobenius algebra. In
general, solutions of the \nybe can be produced from Rota-Baxter
operators, and then from $\calo$-operators when the solutions are
taken in semi-direct product algebras. In the other direction,
Rota-Baxter operators can be obtained from solutions of the \nybe
in unitizations of algebras. Finally a classification is obtained for
solutions of the \nybe satisfying the mentioned invariant
condition in all unital complex algebras of dimensions two and
three. All these solutions are shown to come from
Rota-Baxter operators.
\end{abstract}

\subjclass[2010]{16T25,
17B38,      %Yang-Baxter equations and Rota-Baxter operators, 2020AMS
16W99,
%81T45,   %Topological field theories
17A30,%Nonassociative algebras satisfying other identities
%17B40,  %Automorphisms, derivations, other operators
57R56   %Topological quantum field theories
%47A67   %Representation theory of linear operators, 2020 AMS
}

\keywords{ associative Yang-Baxter equation; classical Yang-Baxter
equation; $\mathcal{O}$-operator; Rota-Baxter operator;
dendriform algebra}

\maketitle

\vspace{-1.2cm}

\tableofcontents

\vspace{-1cm}

\allowdisplaybreaks

\section{Introduction}

The aim of this paper is to give operator forms of the
nonhomogeneous associative classical Yang-Baxter
equation in terms of Rota-Baxter operators and the more general
$\calo$-operators.

\subsection{CYBE, \aybe and their operator forms}

The classical Yang-Baxter equation (CYBE) was first given in the following tensor form
\begin{equation}\notag
[r_{12},r_{13}]+[r_{12},r_{23}]+[r_{13},r_{23}]=0,
\mlabel{eq:cybe}
\end{equation}
where $r\in \g\otimes \g$ and $\g$ is a Lie algebra (see \cite{CP} for details).
The CYBE arose from the study of
inverse scattering theory in 1980s. Later it was recognized as the
``semi-classical limit" of the quantum Yang-Baxter equation which
was encountered by C.~N. Yang in the computation of the
eigenfunctions of a one-dimensional fermion gas with delta
function interactions~\mcite{Ya} and by
R.~J. Baxter in the solution of the eight vertex model in
statistical mechanics~\mcite{BaR}.
The study of the CYBE is also related to classical integrable
systems and quantum groups (see~\cite{CP} and the references
therein).

An important approach in the study of the CYBE was through the
interpretation of its tensor form in various operator forms which
proved to be effective in providing solutions of the CYBE, in
addition to the well-known work of Belavin and
Drinfeld~\mcite{BD}. First Semonov-Tian-Shansky~\mcite{STS} showed
that if there exists a nondegenerate symmetric invariant bilinear
form on a Lie algebra $\g$ and if a solution $r$ of
the CYBE is skew-symmetric, then $r$ can be equivalently expressed
as a linear operator $R:\g\to \g$ satisfying the operator identity
\begin{equation}\label{eq:CYBE-RB}
[R(x), R(y)]=R([R(x),y])+R([x,R(y)]),\;\;\forall x,y\in
\g,\end{equation}
which is then regarded as an {\bf operator form} of the CYBE.
Note that Eq.~(\ref{eq:CYBE-RB}) is exactly the
Rota-Baxter relation (of weight zero) in Eq.~\meqref{eq:rbo} for Lie algebras.

In order for the approach to work more generally, Kupershmidt revisited operator forms of the CYBE
in~\cite{Ku} and noted that, when $r$ is skew-symmetric, the
tensor form of the CYBE is equivalent to a linear map $r:
\g^\ast \rightarrow \g$ satisfying
$$[r(x),r(y)]=r({\rm ad}^\ast  r(x)(y)-{\rm
ad}^\ast r(y)(x)),\;\forall x,y\in \g^\ast,$$
where $\g^\ast $ is the dual
space of $\g$ and ${\rm ad}^\ast $ is the dual representation of
the adjoint representation (coadjoint representation) of the Lie
algebra $\g$. He further generalized the above ${\rm
ad}^\ast $ to an arbitrary representation $\rho:\g\rightarrow
\gl(V)$ of $\g$, that is, a linear map $T:V\rightarrow \g$,
satisfying
$$[T(u), T(v)]=T(\rho(T(u))v-\rho(T(v))u),\quad \forall u,v\in V,$$
which was regarded as a natural generalization of the CYBE. Such
an operator is called an {\bf ${\mathcal O}$-operator} associated
to $\rho$. Note that the operator form (\ref{eq:CYBE-RB}) of
the CYBE given by Semonov-Tian-Shansky is just an ${\mathcal
O}$-operator associated to the adjoint representation of $\g$.

Going in the other direction, any $\mathcal O$-operator gives a
skew-symmetric solution of the CYBE in a semi-direct product Lie
algebra, completing the cycle from
the tensor form to the operator form and back to the tensor form
of the CYBE. Moreover, there is a closely related algebraic
structure called the pre-Lie algebra. Any $\mathcal O$-operator
gives a pre-Lie algebra and conversely, any pre-Lie algebra
naturally gives an $\mathcal O$-operator of the commutator Lie
algebra, and hence naturally gives rise to a solution of the CYBE~\cite{Bai2}.

An analogue of the CYBE for associative algebras is the
{\bf associative Yang-Baxter equation (AYBE)}~\mcite{Ag2}:
\begin{equation}\notag
r_{12}r_{13}+r_{13}r_{23}-r_{23}r_{12}=0,
\end{equation}
for $r\in A\ot A$ where $A$ is an associative algebra (see
Definition~\mref{de:nybe} for details). Its form with spectral
parameters was given in~\mcite{Po} in connection with the CYBE and
the quantum Yang-Baxter equation. The \aybe arose from the study
of the (antisymmetric) infinitesimal bialgebras, a notion traced
back to Joni and Rota in order to provide an algebraic framework
for the calculus of divided differences~\mcite{HR, JR} and, in the
antisymmetric case, carrying the same structures under the names
of ``associative D-bialgebra" in~\mcite{Zh} and ``balanced
infinitesimal bialgebra" in the sense of the opposite algebra in
\cite{Ag2}. The AYBEs have found applications in various fields in
mathematics and mathematical physics such as Poisson
geometry, integrable systems, quantum groups, and mirror
symmetry~\mcite{ARR,KZ,LP,ORS,ORS2,OS,Sc,Sc2,vd}. 

Motivated by the operator approach to the CYBE and the Rota-Baxter
operators with weights, $\calo$-operators with weights were
introduced to give an operator approach to the AYBE~\cite{BGN1},
while a method of obtaining Rota-Baxter operators from solutions
of the (opposite) \aybe was obtained in~\mcite{Ag1}. Briefly
speaking, under the skew-symmetric condition, a solution of the
AYBE is an $\mathcal O$-operator associated to the dual
representation of the adjoint representation, while an $\mathcal
O$-operator gives a skew-symmetric solution of the AYBE in a
semi-direct product associative algebra. Furthermore, the role played by pre-Lie algebras in CYBE is similarly played by dendriform algebras introduced by Loday~\mcite{Lo}, that is, any $\mathcal O$-operator induces a dendriform algebra
structure on the representation space and conversely, a dendriform
algebra gives a natural $\mathcal O$-operator and hence there is a
construction of (skew-symmetric) solutions of the AYBE from
dendriform algebras~\mcite{BGN0,BGN2}. Moreover,  such relationships are generalized to connect the solutions of the AYBE satisfying certain ``invariant" conditions and $\mathcal O$-operators with weights~\cite{BGN0,BGN1}.

In turn, these studies of the \aybe by $\calo$-operators with weights led to the introduction of similar $\calo$-operators on Lie algebras. These generalizations have found fruitful applications to the CYBE and further to Lax pairs, Lie bialgebras, and PostLie algebras~\mcite{BGNp1,BGNp2}.

\subsection{Nonhomogeneous AYBE and its operator form}
The notion of a {\bf non-homo-geneous associative classical
Yang-Baxter equation (nhacYBe)}~\mcite{OP} is the equation (detailed in Definition~\mref{de:nybe})
\begin{equation}\mlabel{eq:nhybe0}
r_{12}r_{13}+r_{13}r_{23}-r_{23}r_{12}=\mu r_{13},
\end{equation}
where $\mu$ is a fixed constant.
Its opposite form, given in
Eq.~\meqref{eq:nhybe-opp}, was called the {\bf associative
classical Yang-Baxter equation of  $\mu$} in~\mcite{EF}. Taking
$\mu=0$ recovers the \aybe.

The \nybe arose from the study of the quantum Yang-Baxter equation and Bezout operators.
Another motivation for introducing the \nybe is the
{\bf $\mu$-infinitesimal bialgebras}, that is, a triple
$(A,\apr,\Delta)$ consisting of an algebra $(A,\apr)$ and a
coalgebra $(A,\Delta)$ satisfying the compatibility
condition
\begin{equation}\mlabel{eq:ibmu}
\Delta(x\apr y)=(L(x)\otimes \id)\Delta(y)+\Delta(x)(\id\otimes
R(y))-\mu x\otimes y,\;\;\forall x,y\in A,
\end{equation}
where $L(x),R(x)$ are left and right multiplication operators of
$(A,\apr)$ respectively.
 When $\mu=1$, it was also called a {\bf
unital infinitesimal bialgebra}~\mcite{LR06} and appeared in
several topics such as combinatorics, operads and pre-Lie
algebras~\mcite{Foi09,Foi10, ZCGL, ZGL}. A solution of the
opposite form of the \nybe in a unital algebra gives a
$\mu$-infinitesimal bialgebra \mcite{EF,OP}. 

Note that while the AYBE has its origin from the CYBE for Lie
algebras, when $\mu \ne 0$, the \nybe does not have a counterpart for Lie algebras since the term $r_{13}$ on the right hand side of Eq.~\meqref{eq:nhybe0} does not make sense for a Lie algebra.

As in the cases of the CYBE and the AYBE, it is
important to study the \nybe  through its operator forms, to give further
understanding on the nature of the equation, and to provide
constructions of its solutions. To address the challenge from the nonhomogeneity, $\calo$-operators and Rota-Baxter operators are generalized and new approaches are introduced (see Remark~\ref{rmk:diff}).

\subsection{Outline of the paper}
We next provide some details of our operator approach of the \nybe which also serve as an outline of the paper.

In Section~\mref{sec:oop}, we first generalize the notion of an
$\calo$-operator whose weight is a scalar to one whose weight is a
binary operation. We then interpret solutions of the
\nybe equivalently in terms of generalized
$\calo$-operators (Theorem~\mref{pro:iff}) and, in the presence of
a symmetric Frobenius algebra, in terms of generalized Rota-Baxter
algebras (Theorem~\mref{cor:RBsystem}). On Frobenius
algebras, such an interpretation also gives a correspondence
between solutions of the \aybe and Rota-Baxter systems introduced
in \mcite{Br}, rather than Rota-Baxter operators by themselves
(Corollary~\mref{co:rbs}). In order to make a connection with the
existing notion of $\calo$-operators and Rota-Baxter operators, we
explore the additional conditions for solutions of the \nybe. As
it turns out, a solution $r$ of the \nybe can be interpreted in
terms of an $\mathcal O$-operator precisely when it
satisfies the {\bf \inv} condition that the \tsymm
\begin{equation}
\mlabel{eq:inv0}
\sr:=r+\sigma(r)-\mu ({\bf 1}\otimes {\bf 1})
\end{equation}
of  $r$ is invariant, where $\sigma$ is the flip map
(Theorem~\mref{thm:key}). Note that the parameter $\mu$ appears in
both the \nybe and the invariant condition, especially as the
scalar multiple of ${\bf 1}\otimes {\bf 1}$ for the latter. As a
special case, the \tsymm of a solution $r$ is zero means
that $(r,-\sigma(r))$ is an associative Yang-Baxter pair in the
sense of \mcite{Br} (Corollary~\mref{co:ybp}).

In Section~\mref{sec:rbo}, we present a close relationship between
the \nybe and Rota-Baxter operators including but exceeding the
known relationships between the skew-symmetric solutions of the
\aybe and
Rota-Baxter operators of weight zero on Frobenius algebras given
in \mcite{BGN1}. In unital symmetric Frobenius algebras, when the
\tsymm is a multiple of the nondegenerate invariant tensor
corresponding to the nondegenerate bilinear form defining the
Frobenius algebra structure, that is, the \tsymm is a
nondegenerate invariant tensor or zero, there is a
characterization of the solutions of the \nybe by Rota-Baxter
operators (Theorem~\mref{cor:Frobenius}). Taking the matrix
algebras gives the correspondence in \mcite{OP} and taking the
trivial \tsymm and $\mu=0$ yields the correspondence in
\mcite{BGN1}. When the \tsymm is degenerate, in one direction,
there is a construction of solutions of the \nybe from Rota-Baxter
operators satisfying its own invariant conditions
(Proposition~\mref{cor:conRB}). Based on such a construction, we
obtain \inv solutions of the \nybe for $\mu \ne 0$ in semi-direct
product algebras from $\mathcal O$-operators of weight zero as
well as from dendriform algebras of Loday~\mcite{Lo}. Note that
these constructions are different from the construction of
solutions of the \aybe from $\mathcal O$-operators given in
\mcite{BGN1} due to the appearance of the new term $\mu({\bf
1}\otimes {\bf 1})$ in Eq.~\eqref{eq:inv0} (Remark~\ref{rmk:diff}). In the other direction, Rota-Baxter
operators can also be obtained from solutions of the \nybe in an
augmented algebra, that is, the unitization of an associative algebra
(Theorem~\mref{thm:nhybrb} and Corollary~\ref{cor:YBE-RB}).

In Section~\mref{sec:exam}, we
give the classification of the \inv solutions of the \nybe for $\mu \ne 0$ in the
unital complex algebras in dimensions two and three.
These examples indicate that the \inv solutions of the \nybe only comprise a small
part of all solutions of the \nybe. Moreover, we also find that all \inv solutions of the \nybe for $\mu \ne 0$ in the unital complex algebras in dimensions two and three are obtained from Rota-Baxter operators.

\smallskip

\noindent
{\bf Notations.} Throughout this paper, we fix a base field
$\bfk$. Unless otherwise specified, all the vector spaces and
algebras are finite dimensional, although some results and notions
remain valid in the infinite-dimensional case. By a
$\bfk$-algebra, we mean an associative algebra over $\bfk$
not necessarily having a unit.

\section{Characterizations of \nybe by generalized $\calo$-operators}
\mlabel{sec:oop} We first recall some basic definitions and facts
that will be used in this paper. We introduce the notion of
generalized $\mathcal O$-operators whose weight is a binary
operation, especially when the binary operations are obtained from
$A$-bimodule $\bfk$-algebras, we recover the notion of $\mathcal
O$-operators of weight $\lambda$. Then we give a general
interpretation of the \nybe in terms of generalized $\mathcal O$-operators, including a correspondence between solutions of the
\nybe with $\mu=0$, that is, the \aybe,
and Rota-Baxter systems~\mcite{Br} on Frobenius algebras. Finally
under the additional invariant condition, this interpretation
gives a correspondence between \inv solutions of the \nybe and
$\mathcal O$-operators with weight $\lambda$.

\subsection{$\mathcal O$-operators and Rota-Baxter operators for bimodules}

We generalize the notions of $\calo$-operators and Rota-Baxter operators from those with scalar weights to the ones with weights given by binary operations. We first briefly recall some background and refer the reader to~\mcite{Bai1,BGN1} for further details.

Let $(A,\apr)$ be a $\bfk$-algebra.
An {\bf $A$-bimodule} is a $\bfk$-module $V$, together with linear
maps $\ell, r:A\rightarrow \End_\bfk(V)$ satisfying
\begin{equation}
\ell(x\apr y)v=\ell(x)(\ell(y)v),vr(x\apr y)=(vr(x))r(y),
(\ell(x)v)r(y)=\ell(x)(vr(y)),\; \forall\; x,y\in A, v\in V.
\notag
\end{equation}
If we want to be more precise, we also denote an $A$-bimodule $V$ by
the triple $(V,\ell,r)$.

Given a $\bfk$-algebra $A=(A,\apr)$ and $x\in A$, define
\begin{align*}
L(x): A\rightarrow A,\, L(x)y=xy; \quad R(x):A \rightarrow A,\, yR(x)=yx, \, \forall\, y\in A
\end{align*}
to be the left and right actions on $A$. We further define
\begin{align*}
L=L_A:A \rightarrow \End_\bfk (A), \, x\mapsto L(x); \quad R=R_A:A \rightarrow \End_\bfk (A), \, x\mapsto R(x),\, \forall\, x\in A.
\end{align*}
Clearly, $(A, L,R)$ is an $A$-bimodule, called the {\bf adjoint $A$-bimodule}.

There is a natural characterization of semi-direct product
extensions of a $\bfk$-algebra $(A,\apr)$ by an $A$-bimodule. Let
$\ell, r: A\to \End_{\bfk}(V)$ be linear maps. Define a
multiplication on $A\oplus V$ (still denoted by $\apr$) by
 \begin{equation}\notag
 (a+u)\apr(b+v):=a\apr b+(\ell(a)v+u r(b)),\quad \forall a, b\in A, u, v \in V.
 \mlabel{eq:2.3}
 \end{equation}
Then as is well-known, $A\oplus V$ is a $\bfk$-algebra, denoted by
$A\ltimes_{\ell,r} V$ and called the {\bf semi-direct product} of
$A$ by $V$, if and only if $(V,\ell,r)$ is an $A$-bimodule.

For a $\bfk$-module $V$ and its dual module $V^\ast :=\Hom_\bfk (V, \bfk)$, the usual pairing between them is given by
\begin{align*}
\langle, \rangle: V^\ast \times V\rightarrow \bfk, \, \langle u^\ast , v\rangle=u^\ast (v),\, \forall\,  u^\ast \in V^\ast , v\in V.
\end{align*}
Identifying $V$ with $(V^\ast )^\ast $, we also use $\langle v, u^\ast \rangle =\langle u^\ast , v\rangle$.

Let $A$ be a $\bfk$-algebra and let $(V,\ell,r)$ be an $A$-bimodule. Define linear maps
$\ell^\ast ,r^\ast :A\rightarrow \End_\bfk(V^\ast )$ by
\begin{equation}\notag
\langle u^\ast \ell^\ast (x),v\rangle=\langle u^\ast ,\ell(x)v\rangle,\; \langle
r^\ast (x)u^\ast ,v\rangle=\langle u^\ast ,vr(x)\rangle,\;\forall x\in A, u^\ast \in
V^\ast , v\in V, \mlabel{eq:dual}
\end{equation}
respectively. Then $(V^\ast ,r^\ast ,\ell^\ast )$ is also an $A$-bimodule, called the {\bf dual $A$-bimodule} of $(V,\ell,r)$.

To give an operator interpretation of solutions of the \nybe, we
generalize the notion of $\calo$-operators
with weights introduced in~\mcite{BGN1} by dropping the condition
that the multiplication $\circ$ on $R$ turns $(R,\circ,\ell,r)$
into an $A$-bimodule $\bfk$-algebra.

\begin{defi} \mlabel{de:goop}
    Let $(A,\apr)$ be a $\bfk$-algebra. Let $(R,\ell,r)$ be an $A$-bimodule and $\rpr$ a binary operation on $R$. A linear map $\aop:R\rightarrow A$ is called an {\bf $\calo$-operator of weight
        $\rpr$ associated to $(R,\ell,r)$} or simply a {\bf \gen $\calo$-operator} if $\alpha$
    satisfies
    \begin{equation}\notag
    \aop(u)\apr
    \aop(v)=\aop(\ell(\aop(u))v)+\aop(ur(\aop(v)))+\aop(u\rpr
    v),\;\;\forall u,v\in R. \mlabel{eq:aop}
    \end{equation}
    In particular, if $(R,\ell,r)=(A,L_A,R_A)$ is the adjoint $A$-bimodule and $\circ$ is a binary operation on $A$, then an $\mathcal O$-operator $\alpha:A\rightarrow A$ of
    weight $\rpr$ associated to the $A$-bimodule $(A,L_A,R_A)$ is called a {\bf
        Rota-Baxter operator of weight $\circ$}. In this case $\alpha$ satisfies
    \begin{equation}\notag
    \alpha(x)\apr \alpha(y)=\alpha(\alpha(x)\apr y)+\alpha(x\apr\alpha(y))+\alpha(
    x\circ y),\;\;\forall x,y\in A.
    \end{equation}
\end{defi}

\begin{ex}
In the definition of Rota-Baxter operators with weight $\rpr$, when $\rpr$ is given by $x\rpr y:=\lambda x\apr y$ for a given $\lambda\in \bfk$, we recover the usual {\bf Rota-Baxter operator of weight $\lambda$}, with its defining operator identity
\begin{equation}\mlabel{eq:rbo}
P(x)\apr P(y)=P(x\apr y)+P(P(x)\apr y)+\lambda P(x\apr y), \quad \forall x, y \in A.
\end{equation}
Here the notion is named after the mathematicians G.-C. Rota~\mcite{Rot} and G. Baxter~\mcite{Bax} for their early work motivated by fluctuation theory in probability and combinatorics, which again appeared in the work of Connes and Kreimer on renormalization of quantum field theory~\mcite{CK} as a fundamental algebraic structure. See~\mcite{Guo1} for further details.
\end{ex}

We separately define a special case that will be important to us.

\begin{defi}\mlabel{de:twoop}
    Let $(A,\apr)$ be a $\bfk$-algebra and $(R,\ell,r)$ be an $A$-bimodule. Let $s:R\to A$ be a linear map. A linear map $\alpha:R\to A$ is called an {\bf $\calo$-operator \rwith $s$} associated to $(R,\ell,r)$  if
    \begin{equation}\notag
    \alpha(u)\apr\alpha(v)=\aop(\ell(\aop)u))v)+\aop(ur(\aop(v)))+\aop(ur(s(v))), \;\; \forall u, v\in R.
    \mlabel{eq:aops}
    \end{equation}
Likewise $\alpha$ is called an {\bf $\calo$-operator \lwith $s$
        associated to $(R,\ell,r)$} when the third term in the above
    equation is replaced by $\aop(\ell(s(u))v)$.

When the $A$-bimodule is taken to be $(A,L_A,R_A)$,
the operator is called the {\bf Rota-Baxter operator \rwith $s$}
(resp. {\bf \lwith $s$}).
\end{defi}

Obviously the operators in Definition~\mref{de:twoop} are the special cases of the operators in Definition~\mref{de:goop} when the binary operation $\circ$ are defined by
$$u\circ v:=ur(s(v)) \text{ (resp. } u\circ v:=\ell(s(u))v), \quad \forall u, v\in R.$$

To recover the notion of $\calo$-operators with scalar weights introduced in~\mcite{BGN1}, we recall a concept combining $A$-bimodules with $\bfk$-algebras~\mcite{ZBG}.

\begin{defi}
Let $(A,\apr)$ be a $\bfk$-algebra with multiplication $\apr$ and
let $(R,\rpr)$ be a $\bfk$-algebra with multiplication $\rpr$. Let
$\ell, r:A\rightarrow \End_\bfk(R)$ be linear maps. We call
$R$ (or the quadruple
$(R,\rpr,\ell,r)$) an {\bf $A$-bimodule $\bfk$-algebra} if
$(R,\ell,r)$ is an $A$-bimodule that is compatible with the
multiplication $\rpr$ on $R$ in the sense that
\begin{equation}\notag
\ell(x)(v\rpr
w)=(\ell(x)v)\rpr w,\;(v\rpr w)r(x)=v\rpr (wr(x)),\;
(vr(x))\rpr w=v\rpr
(\ell(x)w),\;  \mlabel{eq:twoalg}
\end{equation}
for all $x,y\in A, v,w\in R$.
\mlabel{de:bimal}
\end{defi}

Obviously, $(A,\apr, L_A,R_A)$ is an $A$-bimodule $\bfk$-algebra.

In Definition~\mref{de:goop}, when the $A$-bimodule $(R,\ell,r)$ with multiplication $\ast $ is assumed to be an $A$-bimodule $\bfk$-algebra and when $u \circ v=\lambda u\ast  v$ for $\lambda\in\bfk$, we recover the following notion of an $\calo$-operator with weight $\lambda$ in~\mcite{BGN1}:
\begin{defi}
Let $(A,\apr)$ be a $\bfk$-algebra and let $(R,\ast,\ell,r)$ be an $A$-bimodule $\bfk$-algebra. Let $\lambda\in \bfk$. A linear map $\alpha:R\to A$ is called an {\bf $\calo$-operator of weight $\lambda$} associated to $(R,\ast,\ell,r)$ if $\alpha$ satisfies
\begin{equation}\mlabel{eq:oop}\notag
\alpha(u)\apr \alpha(v)=\alpha(\ell(\alpha(u))v)+\alpha(ur(\alpha(v)))
+\lambda \alpha(u\ast v), \quad \forall u, v\in R.
\end{equation}
When $\ast=0$, then $\calo$ is called an {\bf $\mathcal O$-operator} (of weight zero) associate to the $A$-bimodule $(R, \ell, r)$.
\mlabel{de:oop}
\end{defi}

When $R$ is the $A$-bimodule $\bfk$-algebra $(A,L_A,R_A)$ with $u \circ v:=\lambda u \apr v$ for $\lambda\in \bfk$ and the default multiplication $\apr$ of $A$, we recover the notion of a Rota-Baxter operator $P$ of weight $\lambda$ defined in Eq.~\meqref{eq:rbo}.

These structures can be summarized in the commutative diagram
\begin{equation}\notag
\begin{split}
\xymatrix{
\text{Rota-Baxter operators} \atop \text{\lwith/\rwith\ } s\ar@{^{(}->}[rr] \ar@{_{(}->}[d] && \text{Rota-Baxter operators} \atop \text{of weight  }\circ \ar@{_{(}->}[d] && \text{Rota-Baxter operators} \atop \text{of weight }\lambda \ar@{_{(}->}[ll] \ar@{_{(}->}[d]\\
\calo\text{-operators} \atop \text{\lwith/\rwith\ } s \ar@{^{(}->}[rr]&& \calo\text{-operators} \atop \text{of weight }\circ  && \calo\text{-operators} \atop \text{of weight }\lambda  \ar@{_{(}->}[ll]
}
\end{split}
\label{eq:diaga}
\end{equation}
%Another similar structure, called the Rota-Baxter totally compatible dialgebra, can be found in~\mcite{ZBG}.

\subsection{Operator forms of solutions of \nybe}

We recall the notion of the \nybe and give an interpretation of
solutions of the \nybe in terms of the \gen $\calo$-operators
just introduced.

Let $(A,\apr, \bfone)$ be a unital $\bfk$-algebra of which the multiplication
$\apr$ is often suppressed. For $r=\sum_i a_i\otimes b_i\in A\otimes
A$, denote
\begin{equation}
r_{12}:=\sum_i a_i\otimes b_i\otimes {\bf 1},\;\; r_{13}:=\sum_i
a_i\otimes {\bf 1}\otimes b_i,\;\;r_{23}:=\sum_i {\bf 1}\otimes
a_i\otimes b_i.
\mlabel{eq:3ten}
\end{equation}
Then $r_{12}r_{13},r_{13}r_{23},r_{23}r_{12}$ are elements in the
$\bfk$-algebra $A\otimes A\otimes A$.

\begin{defi} Let $A$ be a unital $\bfk$-algebra and let $r\in A\ot A$.
    \begin{enumerate}
        \item $r$ is a solution of the {\bf associative Yang-Baxter
            equation (AYBE)}
        \begin{equation}\mlabel{eq:aybe}
        r_{12}r_{13}+r_{13}r_{23}-r_{23}r_{12}=0
        \end{equation}
        in $A$ if the equation holds with the notation in Eq.~\meqref{eq:3ten}.
        \item
        Fix a $\mu\in \bfk$. $r$ is a solution of the {\bf $\mu$-nonhomogeneous associative Yang-Baxter equation ($\mu$-\nybe)}
        \begin{equation}\mlabel{eq:nhybe}
        r_{12}r_{13}+r_{13}r_{23}-r_{23}r_{12}=\mu r_{13}
        \end{equation}
        in $A$ if the equation holds with the notation in Eq.~\meqref{eq:3ten}.
        \end{enumerate}
    \mlabel{de:nybe}
\end{defi}
The opposite form of Eq.~\meqref{eq:nhybe} is~\mcite{EF}
\begin{equation}
r_{13} r_{12}+r_{23} r_{13}-r_{12} r_{23}=\mu r_{13}. \mlabel{eq:nhybe-opp}
\end{equation}

\begin{defi} \mlabel{de:tsymm}
Let $A$ be a unital $\bfk$-algebra and $\mu\in \bfk$.
Let $r\in A\ot A$. Define the $\mu$-{\bf \tsymm} of $r$ to be
        \begin{equation}\sr:=r+\sigma(r)-\mu ({\bf 1}\otimes {\bf 1}).
        \mlabel{eq:tw}
        \end{equation}
\end{defi}
The prefix $\mu$ in Definitions~\mref{de:nybe} and \mref{de:tsymm} will be suppressed when its meaning is clear from the context.

Let $r\in A\otimes A$. Define linear maps $r^\sharp,
r^{t\sharp}:A^\ast \rightarrow A$ by the canonical bijections
\begin{equation}\mlabel{eq:bij0}\notag
(\underline{\ })^\sharp: A\ot A\cong \Hom_\bfk(A^\ast ,\bfk)\ot A \cong \Hom_\bfk(A^\ast ,A), \, (\underline{\ })^{t\sharp}=(\underline{\ })^\sharp\sigma:
A\ot A \to \Hom_\bfk(A^\ast ,A).
\end{equation}
Explicitly, $r^\sharp$ and $r^{t\sharp}$ are determined by
\begin{equation}\notag
\langle r^\sharp (a^\ast ),b^\ast \rangle=\langle r,a^\ast \otimes
b^\ast \rangle,\;\;\langle r^{t\sharp}(a^\ast ),b^\ast \rangle=\langle r,
b^\ast \otimes a^\ast \rangle,\;\;\forall a^\ast ,b^\ast \in A^\ast .
\end{equation}
With these notations, $r$ is called
{\bf nondegenerate} if the linear map $r^\sharp$ or $r^{t\sharp}$
is a linear isomorphism. Otherwise, $r$ is called {\bf
    degenerate}. Furthermore, $r$ is symmetric if and only if
\begin{equation} \mlabel{eq:rsymm}\notag
\langle r, a^\ast \ot b^\ast\rangle =\langle r, b^\ast\ot a^\ast\rangle, \text{ that is, } \langle r^\sharp(a^\ast ),b^\ast \rangle = \langle r^\sharp(b^\ast ),a^\ast \rangle,
\;\; \forall a^\ast , b^\ast \in A^\ast .
\end{equation}

We now give an operator form of solutions of the \nybe in terms of the generalized $\calo$-operators with weights given by multiplications.
\begin{thm}\mlabel{pro:iff}
Let $(A,\apr,{\bf 1})$ be a unital $\bfk$-algebra. For $r\in A\otimes A$,  let $\sr$ be the \tsymm of $r$ and let $\sr^\sharp:A^\ast \to A$ be the corresponding linear map.
Then the following statements are equivalent.
\begin{enumerate}
\item  The tensor $r$ is a solution of the \nybe in $A$.
\mlabel{it:iffeq}
\item \mlabel{it:iffr}
The following equation holds.
\begin{equation}\mlabel{eq:rsharp}
r^\sharp(a^\ast )\apr r^\sharp
(b^\ast )+r^\sharp(a^\ast L^\ast (r^{t\sharp}(b^\ast )))-r^\sharp
(R^\ast (r^\sharp(a^\ast ))b^\ast )-\mu r^\sharp (\langle {\bf 1},b^\ast \rangle
a^\ast )=0,\;\;\forall a^\ast ,b^\ast \in A^\ast .
\end{equation}
\item \mlabel{it:iffoop} The linear map $r^\sharp$ from $r$ is an
$\calo$-operator \rwith $-\sr^\sharp$ associated to $(A^\ast
,R^\ast ,L^\ast )$.
\item \mlabel{it:iffrt} The following
equation holds.
\begin{equation}\mlabel{eq:rsharpt}
r^{t\sharp}(a^\ast )\apr r^{t\sharp}
(b^\ast )-r^{t\sharp}(a^\ast L^\ast (r^{t\sharp}(b^\ast )))+r^{t\sharp}
(R^\ast (r^{\sharp}(a^\ast ))b^\ast )-\mu r^{t\sharp} (\langle {\bf
1},a^\ast \rangle b^\ast )=0,\forall a^\ast ,b^\ast \in A^\ast .
\end{equation}
\item The linear map $r^{t\sharp}$ from $\sigma(r)$ is an
$\calo$-operator \lwith $-\sr^\sharp$ associated to $(A,R^\ast
,L^\ast )$.
\mlabel{it:iffoopt}
\end{enumerate}
\end{thm}

\begin{proof} Let $r=\sum_ia_i\otimes b_i$ and $a^\ast ,b^\ast ,c^\ast \in
A^\ast $.

\meqref{it:iffeq} $\Longleftrightarrow$ \meqref{it:iffr}.  We have
\begin{eqnarray*}
\langle r_{12}\apr r_{13}, a^\ast \otimes b^\ast \otimes c^\ast \rangle&=&
\sum_{i,j} \langle a_i\apr a_j, a^\ast \rangle\langle b_i,
b^\ast \rangle\langle b_j, c^\ast \rangle= \sum_j \langle
r^{t\sharp}(b^\ast )\apr a_j, a^\ast \rangle \langle b_j, c^\ast \rangle
\\
&=&\langle r^\sharp(a^\ast L^\ast (r^{t\sharp}(b^\ast ))), c^\ast \rangle,
\end{eqnarray*}
\begin{eqnarray*}
\langle r_{13}\apr r_{23}, a^\ast \otimes b^\ast \otimes
c^\ast \rangle&=&\sum_{i,j}\langle a_i, a^\ast \rangle\langle a_j,
b^\ast \rangle\langle b_i\apr b_j, c^\ast \rangle=\sum_j\langle a_j,
b^\ast \rangle \langle r^\sharp(a^\ast )\apr b_j, c^\ast \rangle\\
&=&\langle r^\sharp(a^\ast )\apr r^\sharp(b^\ast ),c^\ast \rangle,\\
\langle -r_{23}\apr r_{12}, a^\ast \otimes b^\ast \otimes
c^\ast \rangle&=&-\sum_{i,j}\langle a_i, a^\ast \rangle\langle a_j\apr b_i,
b^\ast \rangle\langle b_j, c^\ast \rangle=-\sum_j\langle
a_j\apr r^\sharp(a^\ast ),
b^\ast \rangle \langle b_j, c^\ast \rangle\\
&=&\langle -r^\sharp (R^\ast (r^\sharp(a^\ast ))b^\ast ),c^\ast \rangle,\\
\langle -\mu r_{13}, a^\ast \otimes b^\ast \otimes c^\ast \rangle&=&-\mu\sum_i
\langle a_i,a^\ast \rangle\langle {\bf 1},b^\ast \rangle\langle
b_i,c^\ast \rangle=
\langle -\mu r^\sharp(a^\ast ), c^\ast \rangle\langle {\bf 1}, b^\ast \rangle\\
&=&\langle -\mu r^\sharp(\langle {\bf 1},b^\ast \rangle a^\ast ),
c^\ast \rangle.
\end{eqnarray*}
Hence $r$ satisfies Eq.~(\mref{eq:nhybe}) if and only if
Eq.~(\mref{eq:rsharp}) holds.

\meqref{it:iffr} $\Longleftrightarrow$ \meqref{it:iffoop}. From
the definition of the \tsymm of $r$:
$\sr=r+\sigma(r)-\mu(\bfone\ot \bfone)$, we obtain
$$\sr^\sharp(b^\ast )=r^\sharp(b^\ast ) +r^{t\sharp}(b^\ast ) - \mu \langle \bfone, b^\ast \rangle\bfone, \quad \forall b^\ast \in A^\ast $$
and hence
$$r^{t\sharp}(b^\ast )=-r^\sharp(b^\ast ) +\sr^\sharp(b^\ast ) + \mu \langle \bfone, b^\ast \rangle\bfone, \quad \forall b^\ast \in A^\ast .$$
Further $L^\ast (\bfone)$ is the identity map  on $A^\ast $.
Thus Eq.~\meqref{eq:rsharp} is equivalent to
$$r^\sharp(a^\ast )\apr r^\sharp
(b^\ast )-r^\sharp(a^\ast L^\ast (r^{\sharp}(b^\ast )))-r^\sharp
(R^\ast (r^\sharp(a^\ast ))b^\ast )+r^\sharp(a^\ast L^\ast (\sr^\sharp(b^\ast )))=0,\;\;\forall
a^\ast ,b^\ast \in A^\ast ,$$ as needed.

\meqref{it:iffeq} $\Longleftrightarrow$ \meqref{it:iffrt}. Similarly, we have
\begin{eqnarray*}
\langle r_{12}\apr r_{13}, a^\ast \otimes b^\ast \otimes c^\ast \rangle&=&
\sum_j \langle
r^{t\sharp}(b^\ast )\apr a_j, a^\ast \rangle \langle b_j, c^\ast \rangle
=\langle r^{t\sharp}(b^\ast )\apr r^{t\sharp}(c^\ast ), a^\ast \rangle,\\
\langle r_{13}\apr r_{23}, a^\ast \otimes b^\ast \otimes
c^\ast \rangle&=&\sum_j\langle a_i, a^\ast \rangle \langle b_i\apr r^{\sharp}(b^\ast ),c^\ast \rangle
=\langle r^{t\sharp}(R^\ast (r^{\sharp}(b^\ast ))c^\ast ),a^\ast \rangle,\\
\langle -r_{23}\apr r_{12}, a^\ast \otimes b^\ast \otimes
c^\ast \rangle&=&-\sum_{j}\langle a_i, a^\ast \rangle
\langle r^{t\sharp}(c^\ast )\apr b_i,
b^\ast \rangle=-\langle r^{t\sharp}(b^\ast L^\ast (r^{t\sharp}(c^\ast ))),a^\ast \rangle,\\
\langle -\mu r_{13}, a^\ast \otimes b^\ast \otimes c^\ast \rangle&=& \langle
-\mu r^{t\sharp}(c^\ast ), a^\ast \rangle\langle {\bf 1}, b^\ast \rangle
=\langle -\mu r^{t\sharp}(\langle {\bf 1},b^\ast \rangle c^\ast ),
a^\ast \rangle.
\end{eqnarray*}
Hence $r$ satisfies Eq.~(\mref{eq:nhybe}) if and only if Eq.~(\mref{eq:rsharpt}) holds.

\meqref{it:iffrt} $\Longleftrightarrow$ \meqref{it:iffoopt}. The proof is the same as for \meqref{it:iffr} $\Longleftrightarrow$ \meqref{it:iffoop}.
\end{proof}

We now show that the opposite \nybe in Eq.~\meqref{eq:nhybe-opp} also affords an operator form.
\begin{lem}\label{lem:twoversions}\mlabel{lem:equi}
    Let $(A,\apr,{\bf 1})$ be a unital $\bfk$-algebra. Let $r\in A\otimes A$. Then $r$ satisfies
Eq.~\meqref{eq:nhybe} if and only if $\sigma(r)$ satisfies Eq.~\meqref{eq:nhybe-opp}.
\end{lem}
\begin{proof}
Let $r=\sum_i a_i\otimes b_i\in A\otimes A$. Then $r$ satisfies
Eq.~(\mref{eq:nhybe}) if and only if
\begin{equation}
\sum_{i,j} (a_i\apr a_j\otimes b_i\otimes b_j+a_i\otimes a_j\otimes b_i\apr b_j-a_j\otimes a_i\apr b_j\otimes b_i-\mu a_i\otimes {\bf 1}\otimes b_i)=0.
\mlabel{eq:nhybe-s}
\end{equation}
On the other hand, $\sigma(r)=\sum_ib_i\otimes a_i$ satisfies Eq.~(\mref{eq:nhybe-opp}) if and only if
\begin{equation}
\sum_{i,j}(b_i\apr b_j\otimes a_j\otimes a_i+b_j\otimes b_i\otimes a_i\apr a_j-b_i\otimes a_i\apr b_j\otimes a_j-\mu b_i\otimes {\bf 1}\otimes a_i)=0.
\mlabel{eq:nhybe-sopp}
\end{equation}
Let $\sigma_{13}:A\otimes A\otimes A\rightarrow A\otimes A\otimes
A$ be the linear map defined by $\sigma (x\otimes y\otimes
z)=z\otimes y\otimes x$ for any $x,y,z\in A$. It is
straightforward to check that the left hand side of
Eq.~\eqref{eq:nhybe-s} coincides with the $\sigma_{13}$
applied to the left hand side of
Eq.~\eqref{eq:nhybe-sopp}. This completes the proof.
\end{proof}
Then we have
\begin{cor}
\mlabel{pro:iff-opp}
Let $(A,\apr,{\bf 1})$ be a unital $\bfk$-algebra. For $r\in A\otimes A$,  let $\sr$ be the \tsymm of $r$ and let $\sr^\sharp:A^\ast \to A$ be the corresponding linear map.
Then $r$ satisfies  Eq.~\meqref{eq:nhybe-opp} if and only if the linear
map $r^\sharp:A^\ast \to A$ from $r$ is an $\calo$-operator \lwith
$-\sr^\sharp$ associated to $(A^\ast ,R^\ast ,L^\ast
)$.
\end{cor}

\begin{proof}
Since $\sigma(r)^\sharp=r^{t\sharp}$, the conclusion
follows from Theorem~\mref{pro:iff} and Lemma~\mref{lem:equi}.
\end{proof}

\subsection{Operator forms of solutions in a Frobenius algebra}
We now consider the solutions of the \nybe in a Frobenius algebra.

\begin{defi}
    Let $(A,\apr)$ be a $\bfk$-algebra. A tensor $s\in A\ot A$ is called {\bf invariant} if
    \begin{equation}\notag
    ({\rm id}\otimes L(x)-R(x)\otimes{\rm id})s=0,\;\;\forall x\in
    A.\mlabel{eq:invariant}
    \end{equation}
\end{defi}

\begin{lem} \mlabel{le:syin} {\rm (\mcite{BGN1})}
    Let $(A,\apr)$ be a $\bfk$-algebra.
    Let $s\in A\otimes A$ be symmetric. Then the following conditions
    are equivalent.
    \begin{enumerate}
        \item $s$ is invariant.
        \mlabel{it:sin}
        \item\mlabel{it:ssy}
        $s^\sharp$ satisfies
        \begin{equation}\notag
        R^\ast (s^\sharp(a^\ast ))b^\ast =a^\ast L^\ast (s^\sharp(b^\ast )),\quad \forall
        a^\ast ,b^\ast \in A^\ast .\mlabel{eq:dualssy}
        \end{equation}
        \item \mlabel{it:sbi}
        $s^\sharp$ satisfies
        \begin{equation}\mlabel{eq:dualsabi}\notag
        s^\sharp(R^\ast (x)a^\ast )=x\apr s^\sharp(a^\ast ),\quad
        s^\sharp(a^\ast L^\ast (x))=s^\sharp(a^\ast )\apr x,\quad \forall x\in
        A,a^\ast \in A^\ast.
        \end{equation}
    \end{enumerate}
\end{lem}

\begin{rmk}\mlabel{rmk:noninariant}
    For a unital $\bfk$-algebra $(A,{\bf 1})$, it is obvious that ${\bf 1}\otimes {\bf 1}$ is not invariant when
    $\dim A\geq 2$.
\end{rmk}

\begin{defi} \mlabel{de:1.4}
 A bilinear form $\frakB:=\frakB(\; ,\; )$ on a $\bfk$-algebra $(A,\apr )$ is called {\bf invariant} if
\begin{equation}\notag
 \mathfrak{B}(a\apr b, c)=\mathfrak{B}(a, b\apr c),~~\forall~a, b, c\in A.
 \mlabel{eq:1.3}
 \end{equation}
A {\bf Frobenius algebra} $(A,\frakB)$ is a $\bfk$-algebra $A$
with a nondegenerate invariant bilinear form $\frakB(\; ,\; )$. A
Frobenius algebra $(A,\frakB)$ is called {\bf symmetric} if
$\frakB(\; ,\; )$ is symmetric.
\end{defi}

Let $\Iso_\bk(M,N)$ denote the set of linear bijections between
$\bfk$-vector spaces $M$ and $N$ of the same dimension. Let
$\ndhom(A\ot A,\bk)$ and $\ndt(A\ot A)$ denote the set of
nondegenerate bilinear forms on $A$ and nondegenerate tensors in
$A\ot A$ respectively. Then by definition, the linear bijection
$\Hom_\bk(A\ot A,\bk) \cong \Hom_\bk(A,A^\ast )$ restricts to a bijection $\ndhom_\bk(A\ot A,\bk) \cong
\Iso_\bk(A,A^\ast )$. Similarly, the linear bijection $A\ot A\cong
\Hom_\bk(A^\ast ,A)$ restricts to a bijection $\ndt(A\ot A) \cong
\Iso_\bk(A^\ast ,A)$. Then thanks to the bijection
$\Iso_\bk(A,A^\ast )\cong \Iso_\bk(A^\ast ,A)$ by taking inverse,
we obtain a bijection
\begin{equation}
\ndhom_\bk(A\ot A,\bk)\cong \Iso_\bk(A,A^\ast )\cong \Iso_\bk(A^\ast ,A) \cong \ndt(A\ot A).
\mlabel{eq:bij}
\end{equation}
Explicitly, let $\frakB$ be a nondegenerate
bilinear form. Let $\phi^\sharp=\phi^\sharp_\frakB:A^\ast \rightarrow
A$ be the linear isomorphism defined by
\begin{equation}\mlabel{eq:phi}
\langle {\phi^\sharp}^{-1}(x), y\rangle=\frakB(x,y),\forall x,y\in
A.
\end{equation}
The corresponding tensor $\phi\in A\ot A$ is the one induced from
the linear map $\phi^\sharp$.

\begin{lem}
Let $(A,\apr )$ be a $\bfk$-algebra. A nondegenerate bilinear form is symmetric and  invariant $($and hence gives a symmetric Frobenius algebra $(A,\apr ,\frakB)$$)$ if and only if the corresponding
$\phi\in A\ot A$ via Eq.~\eqref{eq:bij} is symmetric and invariant.
\mlabel{lem:Frobenius}
\end{lem}

\begin{proof} For any $a^\ast ,b^\ast \in A^\ast $, let $x=\phi^\sharp(a^\ast )$ and
    $y=\phi^\sharp(b^\ast )$.
Then from Eq.~\eqref{eq:phi} we obtain
$$\frakB(x,y)=\langle (\phi^\sharp)^{-1}(x),y\rangle =\langle a^\ast , \phi^\sharp(b^*)\rangle =\langle b^\ast\ot a^\ast ,\phi\rangle.$$
Thus $\frakB(x,y)-\frakB(y,x)=\langle b^\ast\ot a^\ast-a^\ast\ot b^\ast, \phi\rangle$ which
shows that $\frakB$ is symmetric if and only if $\phi$ is
symmetric.

Then under the symmetric condition of $\frakB$ and hence of $\phi$, for any $z\in A$, we have
\begin{align*}
\frakB(y\apr z,x)-\frakB(y,z\apr x)&= \frakB(\phi^\sharp(b^\ast )\apr z,\phi^\sharp(a^\ast ))- \frakB(\phi^\sharp(b^\ast ),z\apr \phi^\sharp(a^\ast ))\\
&= \langle a^\ast , \phi^\sharp(b^\ast )\apr z\rangle -
\langle b^\ast , z\apr \phi^\sharp(a^\ast )\rangle \\
&=\langle a^\ast L^\ast (\phi^\sharp(b^\ast )), z\rangle
-\langle R^\ast (\phi^\sharp(a^\ast ))b^\ast ,z\rangle\\
&= \langle a^\ast L^\ast (\phi^\sharp(b^\ast )) -\langle R^\ast
(\phi^\sharp(a^\ast ))b^\ast ,z\rangle.
\end{align*}
By Lemma~\mref{le:syin}, this shows that $\frakB$ is symmetric and invariant if and only if $\phi$ is symmetric and invariant.
\end{proof}

\begin{thm}\mlabel{cor:RBsystem}
Let $(A,\apr ,{\bf 1}, \frakB)$ be a unital symmetric Frobenius
algebra. Let $\phi^\sharp:A^\ast \rightarrow A$ be the linear
isomorphism defined by Eq.~\meqref{eq:phi}. For $r\in A\otimes A$, let the linear maps $P_r, P_r^t:A\rightarrow A$ be
defined respectively by
\begin{equation} \mlabel{eq:ppt}
P_r(x):=r^\sharp(\phi^\sharp)^{-1}(x),\;\;P^t_r(x):={r^t}^\sharp(\phi^\sharp)^{-1}(x),\;\;\forall
x\in A.
\end{equation}
Let $\sr^\sharp(a^\ast ):=r^\sharp(a^\ast )+r^{t\sharp}(a^\ast )-\mu\langle \bfone, a^\ast \rangle \bfone, a^\ast \in A^\ast $ be defined by the \tsymm $\sr$ of $r$.
Then the following statements are equivalent.
\begin{enumerate}
\item $r$ is a solution of the \nybe in $A$.
\mlabel{it:rbs-nybe}
\item \mlabel{it:rbs-pb}
The following equation holds.
\begin{equation}\mlabel{eq:rsharp1}
P_r(x)\apr P_r(y)=P_r(P_r(x)\apr y)-P_r(x\apr P_r^t(y))+\mu \frakB({\bf
1},y)P_r(x),\;\;\forall x,y\in A.
\end{equation}
\item \mlabel{it:rbs-pbt}
The following equation holds.
\begin{equation}\mlabel{eq:rsharpt1}
P^t_r(x)\apr P^t_r(y)=P^t_r(-P_r(x)\apr y)+P^t_r(x\apr P_r^t(y))+\mu \frakB({\bf
    1},x)P_r^t(y),\;\;\forall x,y\in A.
\end{equation}
\item
\mlabel{it:rbs-rbrt}
The operator $P_r$ on $A$ is a Rota-Baxter operator \rwith $-\sr^\sharp(\phi^\sharp)^{-1}$, that is,
\begin{equation}\notag
P_r(x)\apr P_r(y)=P_r(P_r(x)\apr y)+P_r(x\apr P_r(y)) - P_r(x\apr  \sr^\sharp
(\phi^\sharp)^{-1}(y)), \;\; \forall x, y \in A.
\end{equation}
\item \mlabel{it:rbs-rblt}
The operator $P_r^t$ on $A$ is a Rota-Baxter operator \lwith $-\sr^\sharp(\phi^\sharp)^{-1}$, that is,
\begin{equation}\notag
P_r^t(x)\apr P_r^t(y)=P_r^t(P_r^t(x)\apr y)+P_r^t(x\apr P_r^t(y)) -
P_r^t(\sr^\sharp (\phi^\sharp)^{-1}(x)\apr y), \;\; \forall x, y \in A.
\end{equation}
\end{enumerate}
\end{thm}

\begin{proof} For any $x,y\in A$, set
$a^\ast ={\phi^\sharp}^{-1}(x),b^\ast ={\phi^\sharp}^{-1}(y)$, we have
\begin{eqnarray*}
P_r(x)\apr P_r(y)&=&r^\sharp(a^\ast )\apr r^\sharp(b^\ast ),\\
P_r(P_r(x)\apr y)&=&
r^\sharp{\phi^\sharp}^{-1}(r^\sharp{\phi^\sharp}^{-1}(x)\apr \phi^\sharp(b^\ast ))=r^\sharp{\phi^\sharp}^{-1}(r^\sharp(a^\ast )\apr \phi^\sharp(b^\ast ))=r^\sharp(R^\ast (r^\sharp(a^\ast ))b^\ast ),\\
P_r(x\apr P_r^t(y))&=& r^\sharp{\phi^\sharp}^{-1}(\phi^\sharp
(a^\ast )\apr r^{t\sharp}{\phi^\sharp}^{-1}(y))=r^\sharp{\phi^\sharp}^{-1}(\phi^\sharp
(a^\ast )\apr r^{t\sharp} (b^\ast ))=r^\sharp(a^\ast L^\ast (r^{t\sharp}(b^\ast ))),\\
\frakB({\bf 1},y)P_r(x)&=&P_r\phi^\sharp(a^\ast )\frakB({\bf 1},y)
=r^\sharp(\langle {\bf 1},b^\ast \rangle a^\ast ).
\end{eqnarray*}
Note that the invariance of $\phi$ given by
Lemma~\mref{lem:Frobenius} is used in deriving
Eqs.~\meqref{eq:rsharp1} and \meqref{eq:rsharpt1}. By
Theorem~\mref{pro:iff}, $r$ satisfies Eq.~(\mref{eq:nhybe}) if and
only if $P_r$ satisfies Eq.~(\mref{eq:rsharp1}). Similarly, we
show that $r$ satisfies Eq.~(\mref{eq:nhybe}) if and only if $P_r^t$
satisfies Eq.~(\mref{eq:rsharpt1}). Hence statements
\meqref{it:rbs-nybe} -- \meqref{it:rbs-pbt} are equivalent.

Next for
any $x\in A$ and $b^\ast \in A^\ast $, we have
\begin{eqnarray*}
\langle P_r(x)+P_r^t(x), b^\ast \rangle&=&\langle r^\sharp
({\phi^\sharp}^{-1}(x))+{r^t}^\sharp ({\phi^\sharp}^{-1}(x)),
b^\ast \rangle\\
&=&\langle \sr^\sharp{\phi^\sharp}^{-1}(x)+\mu \langle
{\phi^\sharp}^{-1}(x),{\bf 1}\rangle {\bf 1},
b^\ast \rangle\\
&=&\langle \sr^\sharp{\phi^\sharp}^{-1}(x)+\mu \frakB(x,{\bf 1}){\bf
1}, b^\ast \rangle.
\end{eqnarray*}
Hence
$$P^t_r(x)=-P_r(x)+\sr^\sharp{\phi^\sharp}^{-1}(x)+\mu \frakB(x,{\bf 1}){\bf 1},\;\;\forall x\in A.$$
Then the equivalence of the statement~\meqref{it:rbs-pb} (resp.  \meqref{it:rbs-pbt}) to the statement~\meqref{it:rbs-rbrt} (resp. \meqref{it:rbs-rblt}) follows from
applying this equation.
\end{proof}

We give an application to Rota-Baxter systems introduced by Brzezi\'{n}ski~\mcite{Br}.

\begin{defi}Let $A$ be a $\bfk$-algebra. Let $P,S:A\rightarrow A$ be two linear maps. The triple $(A,P,S)$ is called
    a {\bf Rota-Baxter system} if for any $x, y\in A$, the following
    equations hold
$$    P(x)P(y)=P(P(x)y+xS(y)),\quad
S(x)S(y)=S(P(x)y+xS(y)).
$$
\end{defi}

Taking $\mu=0$ in the equivalent statements \meqref{it:rbs-nybe} -- \meqref{it:rbs-pbt} in Theorem~\mref{cor:RBsystem} gives
\begin{cor}\mlabel{co:rbs}
Let $(A,\apr ,{\bf 1}, \frakB)$ be a unital symmetric Frobenius
algebra. For $r\in A\ot A$, let $P_r$ and $P_r^t$ be defined as in Eq.~\meqref{eq:ppt}. Then $r$ is a solution of the \aybe in Eq.~\eqref{eq:aybe} if and only
    if $(A,P_r,-P^t_r)$ is a Rota-Baxter system.
\end{cor}

\subsection{Operator forms of \inv solutions of \nybe} We now show that, under an invariant condition, solutions of the \nybe can be interpreted in terms of the usual $\mathcal O$-operators in Definition~\mref{de:oop}.

\begin{defi} \mlabel{de:inv}
    Let $(A,\apr)$ be a $\bfk$-algebra.
A tensor $r\in A\ot A$ is called {\bf \inv} if its \tsymm $\sr$
defined in Eq.~\meqref{eq:tw} is
        invariant.
\end{defi}

\begin{lem}
    \begin{enumerate}
        \item
        Let $(A,\apr, {\bf 1})$ be a unital $\bfk$-algebra. Let $s\in A\otimes A$
        be symmetric and invariant. Set
        \begin{equation}\mlabel{eq:circ}
        a^\ast \circ b^\ast  :=a^\ast L^\ast (s^\sharp(b^\ast ))=R^\ast (s^\sharp(a^\ast ))b^\ast ,\;\;\forall a^\ast ,b^\ast \in A^\ast .
        \end{equation}
        Then $(A^\ast ,\circ,R^\ast ,L^\ast )$ is an $A$-bimodule $\bfk$-algebra.
        \mlabel{it:inv1}
        \item
        Let $(A^\ast ,\circ,R^\ast ,L^\ast )$ be an $A$-bimodule
        $\bfk$-algebra. Define a linear map $s^\sharp:A^\ast \rightarrow A$
        or equivalently $s\in A\otimes A$ by
        \begin{equation}\mlabel{eq:defn}
        \langle s, a^\ast \otimes b^\ast \rangle :=\langle
        s^\sharp(a^\ast ),b^\ast \rangle :=\langle b^\ast \circ a^\ast ,{\bf
            1}\rangle,\;\;\forall a^\ast ,b^\ast \in A^\ast .
        \end{equation}
        Suppose
        \begin{equation}\mlabel{eq:ss1}
        \langle a^\ast \circ b^\ast ,{\bf 1}\rangle=\langle b^\ast \circ a^\ast , {\bf
            1}\rangle,\;\;\forall a^\ast ,b^\ast \in A^\ast ,
        \end{equation}
        and $s^\sharp$ satisfies
        \begin{equation}\mlabel{eq:sss}
        \langle s^\sharp(a^\ast )\apr x, b^\ast \rangle=\langle b^\ast \circ
        a^\ast ,x\rangle,\;\; \forall x\in
        A, a^\ast ,b^\ast \in A^\ast .
        \end{equation}
        Then $s$ is symmetric and invariant.
        \mlabel{it:inv2}
    \end{enumerate}
    \mlabel{lem:invariant}
\end{lem}

\begin{proof}
    \meqref{it:inv1}.
    Let $a^\ast ,b^\ast ,c^\ast \in A^\ast $ and $x,y\in A$. Then we have
    \begin{eqnarray*}
        (a^\ast \circ b^\ast )\circ c^\ast &=& a^\ast L^\ast (s^\sharp(b^\ast ))\circ c^\ast  = a^\ast L^\ast (s^\sharp(b^\ast ))L^\ast (s^\sharp(c^\ast )), \\
        a^\ast \circ(b^\ast \circ c^\ast )&=&a^\ast \circ b^\ast L^\ast (s^\sharp(c^\ast )) = a^\ast L^\ast (s^\sharp(b^\ast L^\ast (s^\sharp(c^\ast )))) = a^\ast L^\ast (s^\sharp(b^\ast )\ast  s^\sharp(c^\ast )).
    \end{eqnarray*}
    Hence $(A^\ast ,\circ)$ is a $\bfk$-algebra. Moreover,
    \begin{eqnarray*}
        \langle R^\ast (x)(a^\ast \circ b^\ast ), y\rangle&=&\langle a^\ast L^\ast (s^\sharp(b^\ast )), y\apr x\rangle =\langle a^\ast , s^\sharp (b^\ast )\apr y\apr x\rangle,\\
        \langle (R^\ast (x)a^\ast )\circ b^\ast , y\rangle&=&\langle R^\ast (x)a^\ast ,s^\sharp(b^\ast )\apr y\rangle  = \langle a^\ast , s^\sharp (b^\ast )\apr y\apr x\rangle.
    \end{eqnarray*}
    Hence $R^\ast (x)(a^\ast \circ b^\ast )=(R^\ast (x)a^\ast )\circ b^\ast $. Similarly, we
    have
    $$(a^\ast \circ b^\ast )L^\ast (x)=a^\ast \circ (b^\ast L^\ast (x)),\;\;(a^\ast L^\ast (x))\circ
    b^\ast =a^\ast  \circ (R^\ast (x)b^\ast ).$$
    Therefore
    $(A^\ast ,\circ,R^\ast ,L^\ast )$ is an $A$-bimodule $\bfk$-algebra.

    \meqref{it:inv2}. Applying Eq.~(\mref{eq:ss1}) gives
    \begin{eqnarray*}
        \langle s, a^\ast \otimes b^\ast \rangle &=& \langle s^\sharp (a^\ast ), b^\ast \rangle= \langle b^\ast \circ a^\ast , {\bf 1} \rangle =  \langle a^\ast \circ b^\ast , {\bf 1} \rangle\\
        &=& \langle s^\sharp (b^\ast ),a^\ast \rangle=\langle s, b^\ast \otimes a^\ast \rangle,\;\;\forall
        a^\ast ,b^\ast \in A^\ast .
    \end{eqnarray*}
    Hence $s$ is symmetric. Since
    $(A^\ast ,\circ,R^\ast ,L^\ast )$ is an $A$-bimodule $\bfk$-algebra, we have
    \begin{eqnarray*}
        \langle x\apr s^\sharp(b^\ast ), a^\ast \rangle&=&\langle s^\sharp(b^\ast ),a^\ast  L^\ast (x)\rangle = \langle (a^\ast  L^\ast (x))\circ b^\ast ,{\bf 1}\rangle = \langle a^\ast \circ (R^\ast (x)b^\ast ), {\bf 1}\rangle  \\
        &=&\langle s^\sharp (R^\ast (x)b^\ast ), a^\ast \rangle, \\
        \langle s^\sharp(b^\ast )\apr x, a^\ast \rangle&=&\langle s^\sharp(b^\ast ),R^\ast (x)a^\ast \rangle = \langle (R^\ast (x) a^\ast ) \circ b^\ast ,{\bf 1}\rangle  =\langle b^\ast \circ (R^\ast (x)a^\ast ), {\bf 1}\rangle  \\
        &=&\langle (b^\ast L^\ast (x))\circ a^\ast , {\bf 1}\rangle  = \langle
        a^\ast \circ (b^\ast L^\ast (x)), {\bf 1}\rangle = \langle s^\sharp
        (b^\ast L^\ast (x)), a^\ast \rangle,
    \end{eqnarray*}
    where $x\in A, a^\ast ,b^\ast \in A^\ast $. Hence $s$ is invariant.
\end{proof}

\begin{rmk}In fact, under the same conditions as for Lemma~\mref{lem:invariant}, Eqs.~(\mref{eq:ss1}) and (\mref{eq:sss}) hold
    if and only if the following equation holds
    \begin{equation}\mlabel{eq:ssse}\notag
    \langle s^\sharp(a^\ast )\apr x, b^\ast \rangle=\langle b^\ast \circ a^\ast ,x\rangle=\langle x\apr s^\sharp(b^\ast ), a^\ast \rangle ,\;\; \forall
    x\in A, a^\ast ,b^\ast \in A^\ast .
    \end{equation}
\end{rmk}

\begin{thm}
Let $(A,\apr, {\bf 1})$  be a unital $\bfk$-algebra. Let $r\in
A\otimes A$ whose \tsymm $\sr$ is invariant.
Let $\circ$ be the binary operation defined from $\sr$ by
Eq.~\meqref{eq:circ}. Then the following statements are
equivalent.
\begin{enumerate}
\item The tensor $r$ is a solution of the \nybe in
Eq.~\meqref{eq:nhybe}. \mlabel{it:key1} \item When $\sr=0$, the
map $r^\sharp$ is an $\mathcal O$-operator of weight zero associated
to the $A$-bimodule $(A^\ast , R^\ast ,L^\ast )$ and when $\sr\ne 0$, the map $r^\sharp$ is an
$\mathcal O$-operator of weight $-1$ associated to the
$A$-bimodule $\bfk$-algebra $(A^\ast ,\circ, R^\ast ,L^\ast
)$.
\mlabel{it:key2} \item
When $\sr=0$, the map $r^{t\sharp}$ is an $\mathcal O$-operator of
weight zero associated to the $A$-bimodule $(A^\ast , R^\ast
,L^\ast )$ and when $\sr\ne 0$,
the map $r^{t\sharp}$ is an $\mathcal O$-operator of weight $-1$
associated to the $A$-bimodule $\bfk$-algebra $(A^\ast ,\circ,
R^\ast ,L^\ast )$.
\mlabel{it:key3}
\end{enumerate}
\mlabel{thm:key}
\end{thm}

\begin{proof}
(\meqref{it:key1} $\Longleftrightarrow$ \meqref{it:key2}). Since $a^\ast \circ b^\ast :=a^\ast L^\ast (\sr^\sharp(b^\ast ))$ and by
Lemma~\mref{lem:invariant}, $(A^\ast ,\circ,R^\ast ,L^\ast )$ is an
$A$-bimodule $\bfk$-algebra, the equivalence follows from Theorem~\mref{pro:iff}.

The proof of (\meqref{it:key1} $\Longleftrightarrow$ \meqref{it:key3}) follows from the same argument.
\end{proof}

\begin{cor}\mlabel{cor:sigma}
Let $(A,\apr, {\bf 1})$  be a unital $\bfk$-algebra. Let $r\in
A\otimes A$ whose \tsymm is invariant. Then $r$ is a solution of the \nybe if and
only if $r$ satisfies Eq.~\meqref{eq:nhybe-opp}.
\end{cor}
\begin{proof}
By Theorem~\mref{thm:key}, the tensor $r$ is a solution the \nybe
if and only if $\sigma(r)$ is a solution of the \nybe, which holds
if and only if $r$ is a solution of
Eq.~\eqref{eq:nhybe-opp} by Lemma~\ref{lem:equi}.
\end{proof}

\begin{rmk} \mlabel{rmk:none} For a unital $\bfk$-algebra $(A, {\bf 1})$, it is obvious that $\mu ({\bf
1\otimes 1})$ is a solution of the \nybe. However, if
$\mu\ne 0$ and $\dim A\geq 2$, then the \tsymm of $\mu ({\bf 1\otimes 1})$ is not invariant (see also
Remark~\mref{rmk:noninariant}).
\end{rmk}

\begin{cor}\mlabel{cor:converse}
Let $(A,\apr,{\bf 1})$ be a unital $\bfk$-algebra and $(A^\ast
,\circ,R^\ast ,L^\ast )$ be an $A$-bimodule $\bfk$-algebra
satisfying Eq.~\meqref{eq:ss1}. Let $s^\sharp:A^\ast \rightarrow
A$ be the linear map from $\circ$ defined by Eq.~\meqref{eq:defn} satisfying
Eq.~\meqref{eq:sss}. Let $P:A^\ast \rightarrow A$ be a linear map
satisfying
\begin{equation}\label{eq:Rcon}
P(a^*)+P^*(a^*)=s^\sharp(a^*)+\mu\langle a^\ast ,{\bf 1}\rangle
{\bf 1},\;\;\forall a^\ast \in A^\ast,
\end{equation}
where $P^\ast:A^\ast\rightarrow A^\ast $ is the dual map of $P$.
Then the following statements are equivalent.
\begin{enumerate}
\item When $s^\sharp=0$, $P$ is an $\mathcal O$-operator of weight
0 associated to $(A^\ast, R^\ast ,L^\ast )$ and when $s^\sharp\ne
0$, $P$ is an $\mathcal O$-operator of weight $-1$ associated to
$(A^\ast ,\circ, R^\ast ,L^\ast )$.
\mlabel{converse1}
\item When $s^\sharp=0$,
$P^\ast$ is an $\mathcal O$-operator of weight zero associated to
$(A^\ast, R^\ast ,L^\ast )$ and when $s^\sharp\ne 0$, $P^\ast$ is
an $\mathcal O$-operator of weight $-1$ associated to $(A^\ast
,\circ, R^\ast ,L^\ast )$.
\mlabel{converse2}
\item The tensor $r\in A\otimes A$
defined by $r^\sharp=P$ is a \inv solution of the \nybe.
\mlabel{converse3}
\item The tensor $r\in A\otimes A$ defined by
$r^{t\sharp}=P$ is a \inv solution of the \nybe.
\mlabel{converse4}
\end{enumerate}
\end{cor}

\begin{proof}
By Lemma~\mref{lem:invariant}, the tensor $s$ from $s^\sharp$ is symmetric and invariant. Set
$P=r^\sharp$. Then for any $a^\ast ,b^\ast \in A^\ast $, we have
\begin{equation}\notag
\langle P(a^\ast )+P^\ast (a^\ast )+ s^\sharp (a^\ast )-\mu\langle a^\ast ,{\bf
1}\rangle {\bf 1}, b^\ast \rangle=\langle r+\sigma (r)+ s-\mu
({\bf 1}\otimes {\bf 1}), a^\ast \otimes b^\ast \rangle.
\mlabel{eq:R-r}
\end{equation}
Hence $P$ satisfies Eq.~(\mref{eq:Rcon}) if and only if the \tsymm of $r$ is symmetric and invariant. By Theorem~\mref{thm:key},
statement~\meqref{converse1} holds if and only if statement~\meqref{converse3} holds. Note that in
this case, $P^\ast =r^{t\sharp}$. Therefore by Theorem~\mref{thm:key}, statement~\meqref{converse2} holds if and only if statement~\meqref{converse1} or
statement~\meqref{converse3} holds.

Furthermore, by the symmetry of $P$ and $P^\ast $, if we set
$P=r^{t\sharp}$, then by the above discussion, we can directly show that
statement~\meqref{converse4} holds if and only if statement~\meqref{converse2} holds.
This proves that all the statements are equivalent.
\end{proof}

We end this subsection with displaying a relationship between
solutions of the \nybe with
trivial \tsymms and associative Yang-Baxter pairs.

\begin{defi} (\mcite{Br})
Let $A$ be a $\bfk$-algebra. {\bf An associative Yang-Baxter pair}
is a pair of elements $r,s\in A\otimes A$ satisfying
\begin{equation}
r_{12}r_{13}-r_{23}r_{12}+r_{13}s_{23}=0, \quad
r_{12}s_{13}-s_{23}s_{12}+s_{13}s_{23}=0.
\notag
\end{equation}
\end{defi}

\begin{pro}{\rm (\mcite{Br})}\mlabel{pro:Br} Let $(A,{\bf 1})$ be a unital $\bfk$-algebra. Let $r,s\in A\otimes A$. If $r-s={\bf 1}\otimes {\bf 1}$, then
the pair $(r,s)$ is an associative Yang-Baxter pair if and only if
$r$ satisfies the \nybe with $\mu=1$.
\end{pro}

\begin{cor} \mlabel{co:ybp}
Let $(A,{\bf 1})$ be a unital $\bfk$-algebra. Let $r\in A\otimes
A$. If
\begin{equation}
r+\sigma(r)=\mu ({\bf 1}\otimes {\bf 1})
\notag
\end{equation}
with $\mu\neq 0$, then $r$ is a solution of the \nybe in Eq.~\meqref{eq:nhybe} if and only if $(r,-\sigma(r))$ is an
associative Yang-Baxter pair.
\end{cor}

\begin{proof}
Let $r\in A\ot A$ be a solution of the \nybe and
$r+\sigma(r)=\mu(\bfone \ot \bfone)$ with $\mu \neq 0$. Then
$r'=\frac{1}{\mu} r$ is a solution of the \nybe with $\mu=1$ and
$r'+\sigma(r')={\bf 1}\otimes {\bf 1}.$ By
Proposition~\mref{pro:Br}, $(r',-\sigma(r'))$ is an associative
Yang-Baxter pair. Hence  $(r,-\sigma(r))$ is an associative
Yang-Baxter pair. Similarly, the converse also holds.
\end{proof}

\section{\Nybe and Rota-Baxter operators}
\mlabel{sec:rbo} In this section, we first give a correspondence
between certain Rota-Baxter operators and \inv solutions of
the \nybe with a specific  \tsymm $\sr$ in unital symmetric
Frobenius algebras. 

When the tensor $\sr$ is degenerate, solutions of the \nybe
in semi-direct product algebras can still be derived from
Rota-Baxter operators, $\calo$-operators and dendriform algebras,
while Rota-Baxter operators can be derived from solutions of the
\nybe in unitization algebras.

\subsection{\Nybe and Rota-Baxter operators on Frobenius algebras}
Extending the correspondence between solutions of the \aybe and
Rota-Baxter systems on Frobenius algebras given in
Corollary~\mref{co:rbs} to the \nybe, we obtain

\begin{thm}\mlabel{cor:Frobenius}
Let $(A,\apr ,{\bf 1},\frakB)$ be a unital symmetric Frobenius
algebra.
Let $\phi^\sharp:A^\ast \rightarrow A$ be the linear
isomorphism from $\frakB$ defined by Eq.~\meqref{eq:phi} and let $\phi\in A\ot A$ be the corresponding invariant symmetric tensor.
Suppose $r\in A\otimes A$ has its \tsymm given by
\begin{equation}
\sr:=r+\sigma(r)-\mu ({\bf 1}\otimes {\bf 1})=-\lambda \phi. \mlabel{eq:phi-inv}
\end{equation}
Define linear maps $P_r, P_r^t:A\rightarrow A$ respectively
by
\begin{equation}
P_r(x):=r^\sharp{\phi^\sharp}^{-1}(x),\;\;P^t_r(x):={r^t}^\sharp{\phi^\sharp}^{-1}(x),\;\;\forall
x\in A.\mlabel{eq:Pr}
\end{equation}
Then the following conditions are equivalent.
\begin{enumerate}
\item $r$ is a solution of the \nybe in $A$.
\mlabel{frob1} \item $P_r$ is a Rota-Baxter operator of weight
$\lambda$, that is, Eq.~\meqref{eq:rbo} holds. \mlabel{frob2}
\item $P^t_r$ is a Rota-Baxter operator of weight $\lambda$.
\mlabel{frob3}
\end{enumerate}
\end{thm}
\begin{proof}
It follows from Theorem~\mref{cor:RBsystem} by taking
$\sr^\sharp=-\lambda \phi^\sharp$.
\end{proof}

A different construction of Rota-Baxter operators from solutions
of the opposite form of the \nybe in
Eq.~(\mref{eq:nhybe-opp}) can be found in~\mcite{EF}.

Taking $\lambda=\mu=0$ in Theorem~\mref{cor:Frobenius}, we obtain the following result.  Note that in this case, $P_r^t=-P_r$.

\begin{cor}
\cite[Corollary 3.17]{BGN1} A skew-symmetric $r\in A\otimes
A$ is a solution of the \aybe in Eq.~\meqref{eq:aybe}
if and only if the linear map $P_r$ defined by Eq.~\meqref{eq:Pr} is a Rota-Baxter operator of
weight zero.
\end{cor}

\begin{ex}\mlabel{ex:fav}
Let $(A,\apr )=(\End_{\bfk} (V),\apr )=(M_n(\bfk),\apr )$ be the matrix
algebra, where $n=\dim V$. It is a Frobenius algebra with the
invariant bilinear form being the trace form, that is,
\begin{equation}
\frakB(x,y):={\rm Tr} (x\apr y), \;\;\forall x,y\in A.\mlabel{eq:trace}
\end{equation}
Take a basis $\{e_1,\cdots,e_n\}$ of $A$ such that
$\frakB(e_i,e_j)=\delta_{ij}$. Let
$$\phi=\sum_i e_i\otimes e_i.$$
Therefore Eq.~(\mref{eq:phi}) holds. Moreover, since
$\End_{\bfk}(V)\otimes \End_{\bfk}(V)\cong\End_{\bfk}(V\otimes
V)$, it is known that $\phi$ is the flip map $\sigma$ on
$V\otimes V$.

Let $r=\sum_ia_i\otimes b_i\in A\otimes A$. Then
$$P_r(x)=r^\sharp {\phi^\sharp}^{-1}(x)=\sum_i \langle
{\phi^\sharp}^{-1}(x), a_i\rangle
b_i=\sum_i\frakB(x,a_i)b_i=\sum_i{\rm Tr}(x\apr a_i)b_i.$$ Similarly,
$P^t_r(x)=\sum_i{\rm Tr}(x\apr b_i)a_i$. Suppose that
$$r+\sigma(r)=-\lambda\sigma+\mu ({\bf 1}\otimes {\bf 1})=-\lambda \phi+\mu
({\bf 1}\otimes {\bf 1}).$$
If $r$ satisfies Eq.~(\mref{eq:nhybe}),
then both $P_r$ and $P^t_r$ are Rota-Baxter operators of weight
$\lambda$. This is exactly the example given in \mcite{OP}.
\end{ex}

\begin{ex} \mlabel{ex:fav-2} We can be more explicit with Example~\mref{ex:fav} when $n=2$. Let $E_{i j}\in M_{2}(\bfk)$, $1\leq i,
j\leq 2$, be the matrix whose $(i,j)$-entry is $1$ and other
entries are zero. Now the matrix algebra $A=M_{2}(\mathbb{C})$ is
a Frobenius algebra with the invariant bilinear form $\frak B$
given by Eq.~(\mref{eq:trace}). An orthogonal basis with respect to the form is
$$e_1=\frac{1}{\sqrt 2}(E_{11}+E_{22}),e_2=\frac{1}{\sqrt
    2}(E_{11}-E_{22}),e_3=\frac{1}{\sqrt 2}(E_{12}+E_{21}),
e_4=\frac{1}{\sqrt {-2}}(E_{12}-E_{21}).$$
Hence the $\phi$ in Example~\mref{ex:fav}
is
\begin{eqnarray*} \phi&=&e_1\otimes e_1+e_2\otimes
    e_2+e_3\otimes e_3+e_4\otimes e_4\\&=& E_{11}\ot E_{11}+E_{22}\ot
    E_{22}+E_{12} \ot E_{21}+E_{21} \ot E_{12}.\end{eqnarray*}

Note that the unit ${\bf 1}$ in $M_{2}(\mathbb{C})$ is
$E_{11}+E_{22}$. Then $${\bf 1}\ot {\bf 1}=E_{11}\ot
E_{11}+E_{11} \ot E_{22}+E_{22}\ot E_{11}+E_{22}\ot E_{22}.$$ On
the other hand, by a direct calculation, we find that
$r = E_{12}\ot E_{21}-E_{11}\ot E_{22}$
is a solution of the \nybe with  $\mu=-1$ in
$M_{2}(\mathbb{C})$. Then we have
\begin{equation*}
r +\sigma (r)=E_{12}\ot E_{21}-E_{11}\ot E_{22}+E_{21}\ot
E_{12}-E_{22}\ot E_{11}=\phi- {\bf 1}\ot {\bf 1}.
\end{equation*}
Hence by Theorem~\mref{cor:Frobenius}, we have a Rota-Baxter
operator ${P}_r$ of weight $-1$ determined by
\begin{equation}\notag
{P}_{r}(E_{11})=-E_{22}, P_{r}(E_{21})=E_{21},
P_{r}(E_{12})=P_{r}(E_{22})=0.\mlabel{eq:weight-1}
\end{equation}
\end{ex}

\subsection{From $\calo$-operators and dendriform algebras to \nybe on semi-direct product algebras}

We now show that $\calo$-operators of weight zero and dendriform
algebras can give rise to
solutions of the \nybe in some semi-direct product algebras.
We first generalize one direction of
Theorem~\mref{cor:Frobenius} by relaxing the condition that the
\tsymm of $r$ is a multiple of a nondegenerate invariant tensor
giving by a symmetric Frobenius algebra.

\begin{pro}\mlabel{cor:conRB}
Let $(A,\apr,{\bf 1})$ be a unital $\bfk$-algebra. Let $s\in A\otimes A$
be symmetric and invariant. Let $P:A\rightarrow A$ be a linear map
satisfying
\begin{equation}\notag
s^\sharp P^\ast (a^\ast )+Ps^\sharp(a^\ast )=-\lambda s^\sharp
(a^\ast )+\mu\langle a^\ast ,{\bf 1}\rangle {\bf 1},\;\;\forall a^\ast \in
A^\ast ,
\mlabel{eq:opform5}
\end{equation}
where $P^\ast $ is the linear dual of $P$.
Let $r_1$ and $r_2$ be defined by
$r_1^\sharp=s^\sharp P^\ast , r_2^\sharp=Ps^\sharp$.
Explicitly, setting $s=\sum_ia_i\otimes b_i$, then
\begin{equation} \mlabel{eq:r12}
r_1:=\sum_iP(a_i)\otimes b_i,\;\;r_2:=\sum_ia_i\otimes P(b_i).
\end{equation}
If $P$ is a Rota-Baxter operator of weight $\lambda$, then $r_1$
and $r_2$ are \inv solutions of the \nybe in $A$.

Conversely, suppose that $s$ is nondegenerate. Let $r\in
A\otimes A$ satisfy
\begin{equation}\notag
r+\sigma(r)=-\lambda s+\mu ({\bf 1}\otimes {\bf 1}).
\end{equation}
Let $P_r, P_r^t:A\rightarrow A$ be the linear maps defined
respectively by
\begin{equation}\notag
P_r(x):=r^\sharp{s^\sharp}^{-1}(x),\;\;P^t_r(x):={r^t}^\sharp{s^\sharp}^{-1}(x),\;\;\forall
x\in A.
\end{equation}
If $r$ is a solution of the \nybe, then $P_r$
and $P_r^t$ are Rota-Baxter operators of weight $\lambda$.
\end{pro}

\begin{proof} In fact, we have $r_2^\sharp={r_1^t}^\sharp$ since
$$\langle {r_1^t}^\sharp (a^\ast ), b^\ast \rangle=\langle s^\sharp
P^\ast (b^\ast ), a^\ast \rangle=\langle s^\sharp (a^\ast ), P^\ast (b^\ast )\rangle=
\langle Ps^\sharp (a^\ast ), b^\ast \rangle=\langle r_2^\sharp(a^\ast ),
b^\ast \rangle,\;\;\forall a^\ast ,b^\ast \in A^\ast .$$
Hence $r_2=\sigma(r_1)$. For any $a^\ast ,b^\ast \in A^\ast $, we have
\begin{eqnarray*}
&&\langle r_1+\sigma (r_1)+\lambda s-\mu ({\bf 1}\otimes {\bf 1}),
a^\ast \otimes b^\ast \rangle\\ &&=\langle s^\sharp
P^\ast (a^\ast ),b^\ast \rangle+\langle s^\sharp P^\ast  (b^\ast ), a^\ast \rangle
+\lambda \langle s^\sharp (a^\ast ),
b^\ast \rangle -\mu \langle {\bf 1}, a^\ast \rangle\langle {\bf 1},b^\ast \rangle\\
&&=\langle s^\sharp P^\ast (a^\ast )+Ps^\sharp(a^\ast )+\lambda s^\sharp
(a^\ast )-\mu\langle a^\ast ,{\bf 1}\rangle {\bf 1},b^\ast \rangle=0.
\end{eqnarray*}
Hence $ r_1+\sigma (r_1)+\lambda s-\mu ({\bf 1}\otimes {\bf
1})=0$. For any $a^\ast ,b^\ast ,c^\ast \in A^\ast $, we have
{\small
\begin{eqnarray*} \langle r_1^\sharp
(a^\ast )\apr r_1^\sharp(b^\ast ),c^\ast \rangle &=& \langle s^\sharp
P^\ast (a^\ast )\apr s^\sharp P^\ast (b^\ast ), c^\ast \rangle=
\langle s^\sharp P^\ast (b^\ast ), c^\ast L^\ast (s^\sharp P^\ast (a^\ast ))\rangle\\
&=&\langle b^\ast , P(s^\sharp (c^\ast )\apr s^\sharp P^\ast (a^\ast ))\rangle\\
&=&\langle b^\ast , -P(s^\sharp (c^\ast )\apr P(s^\sharp(a^\ast )))\rangle+\langle b^\ast , P(-\lambda s^\sharp(c^\ast )\apr s^\sharp(a^\ast )+\mu\langle {\bf 1},a^\ast \rangle s^\sharp(c^\ast ))\rangle,\\
\langle r_1^\sharp(a^\ast L^\ast (r_1^\sharp(b^\ast ))), c^\ast \rangle&=&\langle s^\sharp P^\ast (a^\ast L^\ast (s^\sharp P^\ast (b^\ast ))),c^\ast \rangle
=\langle a^\ast , s^\sharp P^\ast (b^\ast )\apr P(s^\sharp(c^\ast ))\rangle\\&=&\langle a^\ast , s^\sharp(P^\ast (b^\ast )L^\ast P(s^\sharp(c^\ast )))\rangle
=\langle b^\ast , P(P(s^\sharp(c^\ast ))\apr s^\sharp(a^\ast ))\rangle,\\
\langle r_1^\sharp(R^\ast (r_1^\sharp(a^\ast ))b^\ast ),c^\ast \rangle&=&\langle s^\sharp P^\ast (R^\ast (s^\sharp P^\ast (a^\ast ))b^\ast ),c^\ast \rangle
=\langle R^\ast (s^\sharp P^\ast (a^\ast ))b^\ast , P(s^\sharp (c^\ast )\rangle\\
&=&\langle b^\ast ,P(s^\sharp(c^\ast ))\apr s^\sharp P^\ast (a^\ast )\rangle\\
&=&\langle b^\ast , -P(s^\sharp (c^\ast ))\apr P(s^\sharp(a^\ast ))\rangle+\langle b^\ast , -\lambda P(s^\sharp (c^\ast ))\apr s^\sharp(a^\ast )\\
&&+\mu \langle {\bf 1},a^\ast \rangle P(s^\sharp (c^\ast ))\rangle,\\
\langle \lambda r_1^\sharp(a^\ast L\apr (s^\sharp(b^\ast ))),c^\ast \rangle&=&\langle \lambda s^\sharp P^\ast (a^\ast L^\ast (s^\sharp (b^\ast ))), c^\ast \rangle=\langle a^\ast ,
\lambda s^\sharp (b^\ast )\apr Ps^\sharp(c^\ast )\rangle\\&=&\langle a^\ast ,\lambda s^\sharp (b^\ast L^\ast (Ps^\sharp(c^\ast )))\rangle=\langle b^\ast , \lambda P(s^\sharp (c^\ast ))\apr s^\sharp(a^\ast )\rangle.
\end{eqnarray*}}
Hence if $P$ is a Rota-Baxter operator of weight $\lambda$, then
$r_1^\sharp$ is an $\mathcal O$-operator associated to the
$A$-bimodule $\bfk$-algebra $(A^\ast ,\rpr, R^\ast ,L^\ast )$,
where $\rpr$ is defined from $-\lambda s$. Hence $r_1$ is a
solution of the \nybe by Theorem~\mref{thm:key}. By
Theorem~\mref{thm:key} again, $r_2$ is also a solution of the
\nybe since $r_2^\sharp=r_1^{t\sharp}=\sigma(r_1)^\sharp$.

If $s$ is nondegenerate, then from the above proof, it is obvious
that the converse is true. Alternatively, note that when $s$ is nondegenerate, symmetric
and invariant, then it corresponds to a nondegenerate, symmetric
and invariant bilinear form $\frak B$ by
Lemma~\mref{lem:Frobenius} through Eq.~(\ref{eq:phi}) such that
$(A,\frak B)$ is a Frobenius algebra. Then the conclusion follows from Theorem~\mref{cor:Frobenius}.
\end{proof}

\begin{rmk}
When $\mu=0$, the tensor $r_1$ in Eq.~\eqref{eq:r12} recovers a construction in \mcite{EF}.
\end{rmk}

In the rest of this subsection, we provide \inv solutions of the \nybe
 in semi-direct product
algebras from $\calo$-operators of weight zero and dendriform
algebras
by applying Proposition~\mref{cor:conRB}. We first
supply more background.

Let $(A,\apr)$ be a $\bfk$-algebra and $(V,l,r)$ be an
$A$-bimodule. Let $(V^\ast , r^\ast , l^\ast )$ be the dual
$A$-bimodule. Denote the semi-direct  product
algebras
$$\widehat
A:=A\ltimes_{l,r}V, \quad \mathcal
A:=A\ltimes_{r^\ast ,l^\ast }V^\ast.$$ Identify a linear map
 $\beta:V\rightarrow A$  with an element in  $\mathcal
A \otimes \mathcal A$ by the injective map
$$\Hom_\bfk (V, A)\cong A \otimes V^\ast \hookrightarrow \mathcal
A \otimes \mathcal A.$$

\begin{pro} {\rm (\mcite{BGN0})}\mlabel{pro:semi-direct}
Let $A$ be a $\bfk$-algebra and $(V,\ell,r)$ be an $A$-bimodule. Let
$\alpha: V\rightarrow A$ be a linear map. Then $\alpha$ is an
$\mathcal O$-operator of weight zero if and only if the linear map
\begin{equation}\mlabel{eq:alpha}
\widehat\alpha(x,u):=(\alpha(u),-\lambda u),\;\;\forall x\in A, u\in V,
\end{equation}
is a Rota-Baxter operator of weight $\lambda$ on the algebra $\widehat
A$.
\end{pro}

\begin{lem} {\rm (\mcite{BGN1})} \mlabel{lem:semi-direct} Let $(A,\apr )$ be a $\bfk$-algebra and $(V,l,r)$ be an $A$-bimodule.
Let $\beta:V\rightarrow A$ be a linear map. Then $\widetilde
\beta=\beta+\sigma(\beta)\in \mathcal A\otimes \mathcal A$ is
invariant if and only if $\beta$ is a {\bf balanced $A$-bimodule homomorphism}, that is,
\begin{equation}\mlabel{eq:alphabeta}
\beta(l(x)u)=x\apr \beta(v),\;\;\beta(ur(x))=\beta(u)\apr x,\;\;l(\beta(u))v=ur(\beta(v)),\;\;\forall
x\in A,  u,v\in V.
\end{equation}
\end{lem}

\begin{thm}\mlabel{thm:cons} Let $(A,\apr ,{\bf 1})$ be a unital $\bfk$-algebra and $(V,\ell,r)$ be an
$A$-bimodule. Assume that
$\alpha:V\rightarrow A$ is an $\mathcal O$-operator of weight zero
and $\beta:V^\ast \rightarrow A$ is a balanced $A$-bimodule homomorphism.
Let $\widehat \alpha$ be given by Eq.~\meqref{eq:alpha}
and $\widetilde \beta:=\beta+\sigma(\beta)\in \widehat A\otimes \widehat A$.
Let $r_1, r_2\in \widehat{A}\ot \widehat{A}$ be defined by
$$r_1^\sharp:={\widetilde \beta}^\sharp {\widehat\alpha}^\ast,\quad
r_2^\sharp:=\widehat \alpha {\widetilde \beta}^\sharp.$$
If
$\alpha$ and $\beta$ satisfy
\begin{equation}\notag
\beta\alpha^\ast (x^\ast )+\alpha\beta^\ast (x^\ast )=\mu\langle x^\ast ,{\bf
1}\rangle {\bf 1},\;\;\forall x^\ast \in A^\ast,
\end{equation}
then $r_1$ and $r_2$ are \inv solutions of the \nybe in $\widehat A$, with $s=\widetilde
\beta$.
\end{thm}

\begin{proof} By Proposition~\mref{pro:semi-direct}, $\widehat
\alpha$ is a Rota-Baxter operator of weight $\lambda$ of $\widehat A$.
By Lemma~\mref{lem:semi-direct}, $\widetilde \beta\in \widehat A\otimes
\widehat A$ is invariant. Moreover, we have
$${\widehat \alpha}^\ast (x^\ast ,u^\ast )=(0,\alpha^\ast (x^\ast )-\lambda u^\ast ),\;\; {\widetilde \beta}^\sharp
(x^\ast ,u^\ast )=(\beta(u^\ast ),\beta^\ast (x^\ast )),\;\;\forall x^\ast \in A^\ast ,u^\ast \in
V^\ast .$$ Hence for any $x^\ast \in A^\ast ,u^\ast \in V$,  we have
\begin{eqnarray*}
&&{\widetilde \beta}^\sharp {\widehat \alpha}^\ast (x^\ast ,u^\ast )+{\widehat
\alpha}{\widetilde \beta}^\sharp(x^\ast ,u^\ast )+\lambda {\widetilde
\beta}^\sharp(x^\ast ,u^\ast ) -\mu\langle (x^\ast ,u^\ast ),({\bf 1},0)\rangle
({\bf 1},0)\\
&&=(\beta\alpha^\ast (x^\ast )-\lambda \beta
(u^\ast ),0)+(\alpha\beta^\ast (x^\ast ),-\lambda\beta^\ast (x^\ast ))+
\lambda(\beta(u^\ast ),\beta^\ast (x^\ast ))- (\mu\langle
x^\ast ,{\bf 1}\rangle {\bf 1},0)\\
&&=(\beta\alpha^\ast (x^\ast )+\alpha\beta^\ast (x^\ast )-\mu \langle x^\ast , {\bf
1}\rangle {\bf 1},0)=0.
\end{eqnarray*}
By Proposition~\mref{cor:conRB}, the desired result follows.
\end{proof}

\begin{cor}\mlabel{cor:conRB1}
Let $(A,{\bf 1})$ be a unital $\bfk$-algebra. Let $s\in A\otimes
A$ be symmetric and invariant. Let $P:A\rightarrow A$ be a linear
map satisfying
\begin{equation}\notag
s^\sharp P^\ast (a^\ast )+Ps^\sharp(a^\ast )=\mu\langle a^\ast , {\bf 1}\rangle
{\bf 1},\;\;\forall a^\ast \in A^\ast .
\end{equation}
Suppose that $P$ is a Rota-Baxter operator of weight zero.
%Then the following statements hold.
\begin{enumerate}
\item Let $r_1, r_2\in A\ot A$ be defined by
$$r_1^\sharp:=s^\sharp P^\ast , \quad r_2^\sharp:=Ps^\sharp.$$
Then $r_1$ and $r_2$ are \inv solutions of the \nybe
in $A$ whose \tsymms are zero. \item Set $\widehat
A:=A\ltimes_{L,R}A$. Let $\widehat P$ be given by
Eq.~\meqref{eq:alpha} with $\widetilde
{s^\sharp}=s^\sharp+\sigma(s^\sharp)\in \widehat A\otimes \widehat
A$. Let $r_3, r_4\in \widehat{A}\ot \widehat{A}$ be defined by
$$r_3^\sharp:=\big({\widetilde {s^\sharp}}\big)^\sharp {\widehat
    P}^\ast ,\quad r_4^\sharp:=\widehat P \big({\widetilde {s^\sharp}}\big)^\sharp.$$
Then $r_3$ and $r_4$ are \inv solutions of the \nybe in
$\widehat A$ with $s=\widetilde {s^\sharp}$.
\end{enumerate}
\end{cor}

\begin{proof}
(a) follows from Proposition~\mref{cor:conRB} with $\lambda=0$.

(b) follows from Theorem~\mref{thm:cons} where
$(V,l,r)=(A,L,R)$ and $P=\alpha$, $\beta=s^\sharp$. Note that in
this case, if $s$ is invariant and symmetric, then $s^\sharp$
is a balanced $A$-module homomorphism, that is, $s^\sharp$ satisfies Eq.~(\mref{eq:alphabeta}).
\end{proof}

\begin{cor}\mlabel{thm:cons1} Let $(A,\apr, {\bf 1})$ be a unital $\bfk$-algebra. Set $\widehat A:=A\ltimes_{R^\ast ,L^\ast }A^\ast $. Assume that
$\beta:A\rightarrow A$ is a linear map satisfying
\begin{equation}\mlabel{eq:beta}
\beta(x\apr y)=\beta(x)\apr y=x\apr \beta(y),\;\;\forall x,y\in A.
\end{equation}
Let $\mR:A^\ast \rightarrow A$ be an $\mathcal O$-operator of weight
zero associated to $(A^\ast ,R^\ast ,L^\ast )$. Let $\widehat \mR$ be given by
Eq.~\meqref{eq:alpha} and $\widetilde
\beta=\beta+\sigma(\beta)\in \widehat A\otimes \widehat A$.
Let $r, r'\in \widehat{A}\ot \widehat{A}$ be defined by
$$r^\sharp:={\widetilde \beta}^\sharp {\widehat \mR}^\ast , \quad
r'^\sharp:=\widehat \mR {\widetilde \beta}^\sharp.$$
If $\mR$ and $\beta$ satisfy
\begin{equation}\notag
\beta \mR^\ast (x^\ast )+\mR\beta^\ast (x^\ast )=\mu\langle x^\ast ,{\bf 1}\rangle {\bf
1},\;\;\forall x^\ast \in A^\ast ,
\end{equation}
then $r$ and $r'$ are \inv solutions of the \nybe in $\widehat{A}$, when taking $s=\widetilde \beta$.
In particular, suppose
that $\beta=\id$. Then $\beta$ satisfies
Eq.~\meqref{eq:beta}. Suppose that
\begin{equation}\mlabel{eq:Rconno}\notag
\mR(x^\ast )+\mR^\ast (x^\ast )=\mu\langle x^\ast ,1\rangle1,\;\;\forall x^\ast \in A^\ast .
\end{equation}
%Then the following statements hold.
\begin{enumerate}
\item
Let $r_1, r_2\in \widehat{A}\ot \widehat{A}$ be defined by
$$r_1^\sharp:={\widetilde \id}^\sharp {\widehat \mR}^\ast , \quad
r_2^\sharp:=\widehat \mR {\widetilde \id}^\sharp.$$
Then $r_1$ and $r_2$ are \inv solutions of the \nybe in $\widehat{A}$ with $s=\widetilde \id$.
\item Let $r_3,r_4\in A\ot A$ be defined by
$$r_3^\sharp:=\mR,\quad  r_4^\sharp:=\mR^\ast .$$
Then $r_3$ and $r_4$ are \inv solutions of the \nybe in $A$.
\end{enumerate}
\end{cor}

\begin{proof}
The first half part follows from Theorem~\mref{thm:cons} by taking $(V,l,r):=(A^\ast ,R^\ast ,L^\ast )$. Note that in
this case, Eq.~(\mref{eq:alphabeta}) is exactly
Eq.~(\mref{eq:beta}).

(a) follows from the above proof in the case when $\beta=\id$.

(b) follows from Corollary~\mref{cor:converse} in the case
that the \tsymm is zero.
\end{proof}

We finally provide solutions of the \nybe from dendriform algebras.

\begin{defi} \mcite{Lo} Let $A$ be a vector space with two bilinear products denoted by $\prec$ and $\succ$. Then $(A, \prec, \succ)$ is called a {\bf dendriform algebra} if for all $a, b, c\in A$,
\small{
$$
 (a\prec b)\prec c=a\prec (b\prec c+b\succ c),
 (a\succ b)\prec c=a\succ (b\prec c),
 (a\prec b+a\succ b)\succ c=a\succ (b\succ c).
$$
}
\end{defi}

 Let $(A,\prec,\succ)$ be a dendriform algebra.
For any $a\in A$, let $L_{\prec}(a)$, $R_{\prec}(a)$ and
$L_{\succ}(a)$, $R_{\succ}(a)$ denote the left and right
multiplication operators on $(A,\prec)$ and $(A,\succ)$,
respectively.
 Furthermore, define linear maps
 $$R_{\prec}, L_{\succ}: A\to  \End_{\bfk}(A), \quad
 a\mapsto R_{\prec}(a),~a\mapsto L_{\succ}(a),\;\;\forall a\in A.
 $$

As is well-known, for a dendriform algebra $(A,\prec,\succ)$, the
multiplication
\begin{equation}\notag
a\star b:=a\prec b+a\succ b,\quad \forall a,b\in A, \mlabel{eq:5.41}\end{equation}
defines a $\bfk$-algebra $(A,\star )$, called the {\bf associated
algebra} of the dendriform algebra. Moreover, $(A,L_\succ,
R_\prec)$ is a bimodule of the algebra
$(A,\star )$~\mcite{Bai1,Lo}.

A {\bf unital dendriform algebra}~\mcite{Foi07} is a $\bfk$-module $A:=\bfk \bfone \oplus A^+$ such that $(A^+,\prec,\succ)$ is a dendriform algebra and the operations $\prec$ and $\succ$ are extended (partially) to $A$ by
$$ x\prec \bfone =\bfone \succ x= x, \quad x\succ \bfone =\bfone \prec x=0,
 \quad \forall x\in A^+.$$
Note that $\bfone \prec \bfone$ and $\bfone \succ \bfone$ are not
defined. Then $(A,\star ,\bfone)$ is a unital $\bfk$-algebra.

\begin{cor}
Let $(A,\prec,\succ, {\bf 1})$ be a unital dendriform algebra with
the unit {\bf 1}. Let $(A,\star )$ be the associated unital
$\bfk$-algebra with the unit $ {\bf 1}$. Suppose that there is a
linear map $\beta:A^\ast \rightarrow A$ satisfying
\begin{equation}\notag
\beta(R_\prec^\ast (x)y^\ast )=x\star \beta(y^\ast ),\;\beta(y^\ast L_\succ^\ast (x))=\beta(y^\ast )\star x,\;R_\prec^\ast (\beta(y^\ast ))z^\ast =y^\ast L_\succ^\ast (\beta(z^\ast )),\
\end{equation}
for any $ x\in A, y^\ast ,z^\ast \in A^\ast $. Set $\widehat
A=A\ltimes_{L_\succ,R_\prec}A$. Let $\widehat \id$ be given by
Eq.~\meqref{eq:alpha}, that is,
$$\widehat \id(x,y)=(y,-\lambda y),\;\;\forall x,y\in A,$$
 and $\widetilde
\beta=\beta+\sigma(\beta)\in \widehat A\otimes \widehat A$. If in
addition, $\beta$ satisfies
\begin{equation}\notag
\beta(x^\ast )+\beta^\ast (x^\ast )=\mu\langle x^\ast ,{\bf 1}\rangle {\bf
1}, \quad \forall x^\ast \in A^\ast ,
\end{equation}
then $r_1$ and $r_2$ defined by
$$r_1^\sharp:={\widetilde \beta}^\sharp {\widehat\id}^\ast ,\quad  r_2^\sharp:=\widehat
\id{\widetilde \beta}^\sharp$$ are \inv solutions of the \nybe in
$\widehat{A}$, with $s=\widetilde \beta$.
\end{cor}

\begin{proof}
Note that the identity map ${\rm id}$ is an $\mathcal O$-operator
of the associated algebra $(A,\star )$ associated to the bimodule
$(A,L_\succ,R_\prec)$. Hence the conclusion follows from
Theorem~\mref{thm:cons}.
\end{proof}

\begin{rmk}\label{rmk:diff}
The above constructions of \inv solutions of the \nybe are
different from the construction of solutions of the \aybe from
$\mathcal O$-operators given in \mcite{BGN1}, where the symmetric
invariant tensors appearing in the symmetric parts of solutions in the semi-direct product algebras can be ``lifted" from
linear maps from the bimodules to the $k$-algebras themselves as
Lemma~\mref{lem:semi-direct} illustrates. However, it is not true
for the symmetric tensor ${\bf 1}\otimes {\bf 1}$ any more, that
is, the approach in \mcite{BGN1} does not apply here due to the
appearance of the new term $\mu({\bf 1}\otimes {\bf 1})$.
\end{rmk}

\subsection{From \nybe to Rota-Baxter operators on unitization algebras}
We end the section with constructions of Rota-Baxter operators
from solutions of the \nybe in unitization algebras, or
equivalently, augmented algebras.

The {\bf unitization} of a not necessarily unital $\bfk$-algebra $A'$ is the direct sum $\bfk$-algebra $A:=\bfk\oplus A'$. An {\bf
augmentation map} on a unital $\bfk$-algebra $(A,\apr,\bfone)$ is a $\bfk$-algebra homomorphism $\vap: A\rightarrow \bfk$.
An {\bf augmented unital $\bfk$-algebra} is a unital $\bfk$-algebra $(A,\apr ,{\bf 1})$ with an augmentation map $\vap$.

As it is well known~\cite[Theorem~5.1.1]{Co}, augmented unital $\bfk$-algebras are precisely the unitizations of (not necessarily unital) algebras given by
$$ \bfk\oplus A' \longleftrightarrow (A,\vep),$$
where $A:=\bfk\oplus A'$, $\vep$ is the projection to $\bfk$, while $A'$ is $\ker \vep$.

\begin{rmk}\label{rmk:counit}
For an augmented unital $\bfk$-algebra $(A,\apr ,{\bf
1}, \vap)$ with augmentation map $\vap$, there is a
basis $\{e_1,\cdots,e_n\}$ of $A$ such that $e_1={\bf 1}$ and
$\{e_2,\cdots,e_n\}$ is a basis of $\ker \vap=A'$. Let $\{e_1^\ast
,\cdots,e_n^\ast \}$ be the dual basis. Then $\vap=e_1^\ast $.
\end{rmk}

The following conclusion is obvious.

\begin{lem} Let $(A,\apr ,{\bf 1})$ be a unital $\bfk$-algebra and
$\vap$ be an augmentation map. Then $\vap({\bf 1})=1_\bfk$, and
\begin{equation}
\vap(x\apr y\apr z) = \vap(y\apr z\apr x) = \vap(z\apr x\apr y)=\vap(x)\vap(y)\vap(z),\,
\forall x, y, z\in A. \mlabel{eq:cyc}
\end{equation}
\end{lem}

Let $(A, \apr ,{\bf 1},\vap)$ be an augmented unital $\bfk$-algebra.
Define linear maps $$\vap_l:A\otimes A\rightarrow \bfk\otimes
A, \vap_r:A\otimes A\rightarrow A\otimes \bfk$$ respectively by
$$\vap_l:=\vap\otimes \id,\quad \vap_r:=\id\otimes \vap.$$
Similarly, define linear maps $$\vap_{12}:A\otimes A\otimes
A\rightarrow \bfk\otimes \bfk \otimes A, \vap_{23}:A\otimes
A\otimes A\rightarrow A\otimes \bfk\otimes \bfk,
\vap_{13}:A\otimes A\otimes A\rightarrow \bfk\otimes A\otimes \bfk
$$ respectively by
$$\vap_{12}:=\vap\otimes \vap\otimes \id,\quad \vap_{23}:= \id \ot \vap\ot \vap, \quad  \vap_{13}:=
\vap \ot \id\ot \vap.$$ Denote the natural isomorphisms of
algebras~\mcite{Guo1} $$\beta_\ell:\bfk\ot A \to A , 1_{\bfk}\ot a
\mapsto a;\;\; \beta_r: A\ot \bfk\to A , x \ot 1_{\bfk} \mapsto x,
\, \forall\, x\in A.$$
Similarly, define natural isomorphisms of algebras
\begin{align*}
\beta_{12}: &\ \bfk\ot \bfk\ot A \to A , \quad 1_{\bfk}\ot 1_{\bfk} \ot x \mapsto x,\\
\beta_{23}: &\ A \ot \bfk \ot \bfk \to A , \quad x \ot 1_{\bfk}\ot 1_{\bfk} \mapsto x, \\
\beta_{13}: &\ \bfk\ot A\ot \bfk\to A , \quad 1_{\bfk}\ot x \ot
1_{\bfk} \mapsto x , \, \forall\, x\in A.
\end{align*}
For any $x\in A$, set
\begin{equation}\notag
x_{(l)}:=x\otimes {\bf 1}\in A\otimes A,\;\; x_{(r)}:={\bf 1}\otimes
x\in A\otimes A,\end{equation}
\begin{equation}\notag
\loc{x}{1}:= x \ot {\bf 1}\ot {\bf 1} \in A\otimes A\otimes A,  \loc{x}{2}:= {\bf 1}\ot x\ot {\bf 1}\in A\otimes A\otimes
A, \loc{x}{3}:= {\bf 1}\ot {\bf 1} \ot x\in A\otimes A\otimes
A. \mlabel{eq:e3}
\end{equation}

\begin{thm}\mlabel{thm:nhybrb}
Let $(A, \apr, {\bf 1},\vap)$ be an augmented unital
$\bfk$-algebra. Let $r=\sum_{i}a_i \ot b_i \in A\ot A$ be a
solution of the \nybe and $\sr$ be the \tsymm of $r$.
Define linear maps $P,P':A\rightarrow A$ by
\begin{equation}\mlabel{eq:RB}
P(x):= \sum_i\vap(a_i\apr x)b_i,\;\;P'(x) := \sum_i
\vap(b_i\apr x)a_i,\;\;\forall x\in A.
\end{equation}
\begin{enumerate}
\item If $\sr$ is nonzero and satisfies
\begin{equation}
\beta_l(\vap_l ( \sr \apr x_{(l)}))=x,\;\;\forall x\in A, \mlabel{eq:sequation}
\end{equation}
then $P$ and $P'$ are Rota-Baxter operators  of weight
$-1$.  \mlabel{it:nhybrba}
\item
If  $\sr=0$,
then $P$ and $P'$ are Rota-Baxter operators  of weight zero. \mlabel{it:nhybrbb}
\end{enumerate}
\end{thm}

\begin{proof}
(\mref{it:nhybrba}).
Let $x, y\in A$. By definition, we have
\begin{eqnarray}
P(x)&=&\beta_l\vap_l(r\apr x_{(l)}) =\beta_{13}(\vap_{13}
(r_{12}\apr \loc{x}{1}))  \label{eq:RBP22}
= \beta_{12}(\vap_{12}(r_{13}\apr\loc{x}{1})) =
\beta_{12}(\vap_{12}(r_{23}\apr\loc{x}{2})),\ \ \\
P'(x)&=&\beta_r\vap_r(r\apr x_{(r)})=\beta_l\vap_l(\sigma (r)\apr x_{(l)})\nonumber\\
&=&\beta_{23}(\vap_{23} (r_{12}\apr\loc{x}{2})) = \beta_{23}
(\vap_{23}(r_{13}\apr\loc{x}{3})) =
\beta_{13}(\vap_{13}(r_{23}\apr \loc{x}{3})).\label{eq:RBP2}
\end{eqnarray}
Since $r$ satisfies Eq.~(\mref{eq:nhybe}), we have
$$r_{12}\apr r_{13} \apr\loc{x}{1}\apr\loc{y}{2} + r_{13}\apr r_{23}\apr \loc{x}{1}\apr \loc{y}{2} - r_{23}\apr r_{12}\apr \loc{x}{1}\apr \loc{y}{2} =  \mu r_{13} \apr \loc{x}{1}\apr \loc{y}{2}. $$
Applying $\beta_{12}\vap_{12}:A\otimes A\otimes A\rightarrow A$ to
both sides of the above equation, we get
\begin{equation}
\beta_{12}\vap_{12}\big(r_{12}\apr r_{13}
\apr \loc{x}{1}\apr \loc{y}{2}+r_{13}\apr r_{23}\apr\loc{x}{1}\apr\loc{y}{2} - r_{23}\apr r_{12} \apr\loc{x}{1}\apr\loc{y}{2}\big)
= \mu \beta_{12}\big(\vap_{12}( r_{13}
\apr\loc{x}{1}\apr\loc{y}{2})\big). \mlabel{eq:nh00}
\end{equation}
Furthermore, we have
{\small
\begin{equation}
\begin{aligned}
\beta_{12}\big(\vap_{12}(r_{12}\apr r_{13}
\apr\loc{x}{1}\apr\loc{y}{2})\big)
=&\ \beta_{12}(\vap_{12} (\sum_{i,j}  (a_i\apr a_j\apr x) \ot (b_i\apr y)\otimes b_j))\\
=&\ \beta_{12}(\sum_{i,j}  \vap(a_i\apr a_j\apr x) \ot \vap(b_i\apr y)\otimes b_j)\\
=&\ \sum_{i,j} \vap(a_i\apr a_j\apr x)  \vap(b_i\apr y)b_j \\
\stackrel{(\mref{eq:RB})}{=}&\ \sum_{j} \vap(P'(y)\apr a_j\apr x)b_j \\
\stackrel{(\mref{eq:cyc})}{=}&\ \sum_{j} \vap(a_j\apr x \apr P'(y))b_j   \\
\stackrel{(\mref{eq:RB})}{=}& P(x\apr P'(y)).
\end{aligned}
\mlabel{eq:nh1}
\end{equation}
}
Similarly, we have
\begin{eqnarray}
&&\beta_{12}\big(\vap_{12} (r_{13}\apr r_{23}
\apr\loc{x}{1}\apr\loc{y}{2})\big)=P(x)\apr P(y),\mlabel{eq:nh2}\\
&&\beta_{12}\big(\vap_{12}(r_{23}\apr r_{12}\apr \loc{x}{1}\apr\loc{y}{2})\big) =P(P(x)\apr y),\mlabel{eq:nh3}\\
&&\beta_{12}\big(\vap_{12}( r_{13} \apr\loc{x}{1}\apr\loc{y}{2})\big)
=\vap(y) P(x).\mlabel{eq:nh4}
\end{eqnarray}
Substituting Eqs.~(\mref{eq:nh1})-(\mref{eq:nh4}) into
Eq.~(\mref{eq:nh00}) gives
\begin{equation}
P(x)\apr P(y) + P(x\apr P'(y)) - P(P(x)\apr y) = \mu \vap(y)P(x).
\mlabel{eq:ppea}
\end{equation}
Since the \tsymm $\sr$ of $r$ is nonzero, we have
$$\beta_l\vap_l((r+
\sigma(r))\apr x_{(l)}-\mu
x_{(l)})=\beta_l\vap_l( \sr\apr x_{(l)}). $$ By Eqs.~(\mref{eq:RBP22}),~(\mref{eq:RBP2}) and
Eq.~(\mref{eq:sequation}), we obtain
\begin{equation}
P'(x) =   x+\mu \vap(x){\bf 1} - P(x). \mlabel{eq:ppa}
\end{equation}
Substituting Eq.~(\mref{eq:ppa}) into Eq.~(\mref{eq:ppea}) yields
\begin{align*}
&P(x)\apr P(y) + P\Big( x\apr\big(  y+\mu \vap(y){\bf 1} - P(y)\big) \Big) - P(P(x)\apr y) \\
=& P(x)\apr P(y)  + P(x\apr y)+ \mu \vap(y) P(x) - P(x\apr P(y))-
P(P(x)\apr y)\\
 =& \mu \vap(y)P(x),
\end{align*}
that is,
$$P(x)\apr P(y) = P(P(x)\apr y) + P(x\apr P(y))- P(x\apr y),$$
as required. Similarly, we prove that $P'$ is also a Rota-Baxter
operator of weight $-1$.

(\mref{it:nhybrbb}). By an argument similar to the proof of Item (\mref{it:nhybrba}), we also have
\begin{equation}
P(x)\apr P(y) + P(x\apr P'(y)) - P(P(x)\apr y) = \mu \vap(y)P(x).
\mlabel{eq:ppea2}
\end{equation}
Since the \tsymm of $r$ is zero, we obtain
\begin{align*}
r+\sigma (r)-\mu ({\bf 1} \ot  {\bf 1})=0,
\end{align*}
and so
$$\beta_l\vap_l((r+
\sigma(r))\apr x_{(l)}-\mu
x_{(l)})=0.$$
By Eqs.~(\mref{eq:RBP22})-(\mref{eq:RBP2}), we have
\begin{equation}
P'(x) = \mu \vap(x){\bf 1} - P(x). \mlabel{eq:ppa2}
\end{equation}
Substituting Eq.~(\mref{eq:ppa2}) into Eq.~(\mref{eq:ppea2}) shows that $P$ is a Rota-Baxter operator of weight zero. A similar argument  proves that $P'$ is a Rota-Baxter operator of weight zero.
\end{proof}

\begin{cor}\mlabel{coro:rbp0} Let $(A, \apr , {\bf 1},\vap)$ be an augmented unital
$\bfk$-algebra. Let $r\in A\ot A$ be anti-symmetric $($i.e.
$r+\sigma(r)=0)$. If $r$ satisfies the \aybe,
then the operator $P$ defined by Eq.~\meqref{eq:RB} is a Rota-Baxter
operator of weight zero.
\end{cor}

\begin{proof}
    It follows from Theorem~\mref{thm:nhybrb} (\mref{it:nhybrbb}) by taking
    $\mu=0$.
\end{proof}

\begin{cor} \mlabel{cor:YBE-RB}
With the conditions in Theorem~\mref{thm:nhybrb}, suppose that $\sr\in A\ot A$ is nonzero and invariant, that is,
$\sr\apr x_{(l)}=x_{(r)}\apr \sr,\;\;\forall x\in A.$
As in Remark~\mref{rmk:counit}, let $\{e_1={\bf 1},
e_2,\cdots,e_n\}$ be a basis of $A$ and $\{e_1^\ast ,e_2^\ast
,\cdots,e_n^\ast \}$ be the dual basis such that $\vap=e_1^\ast $.
Moreover, suppose
$$\sr={\bf 1}\otimes {\bf 1}+\sum_{i,j>1}s_{ij}e_i\otimes e_j.$$
Then linear maps $P$ and $P'$ defined by Eq.~\meqref{eq:RB} are
Rota-Baxter operators of weight $-1$.
\end{cor}

\begin{proof}
For all $x\in A$, we have
$$\beta_l\vap_l(\sr\apr x_{(l)})=\beta_l\vap_l(x_{(r)}\apr \sr)=\beta_l(\vap({\bf
1})\otimes x)+\sum_{i,j>1}\beta_l(s_{ij}\vap(e_i)\otimes (x\apr e_j))=x,
$$
that is, $\sr$ satisfies Eq.~(\mref{eq:sequation}). Hence
the conclusion follows from Theorem~\mref{thm:nhybrb}.
\end{proof}

\begin{pro}\mlabel{pp:augfrob}
Let $(A,\apr ,{\bf 1})$ be a unital $\bfk$-algebra. If $\vap:
A\rightarrow \bfk$ is an augmentation map, then the bilinear form
$\frak B$ on $A$ defined by
\begin{equation} \frak
B(x,y):=\vap(x)\vap(y),\,\;\forall x,y\in
A,\mlabel{eq:a-B}\end{equation} is symmetric and invariant.
Moreover, $\frak B$ satisfies
$$\frak B(x\apr y,z)=\frak B(y\apr x, z),\;\;\forall x,y,z\in A.$$
In particular, if $\frak B$ is nondegenerate, then $A$ is
commutative. Conversely, if $\frak B$ is a symmetric invariant
bilinear form satisfying
\begin{equation}\notag
\frak B(x,y)=\frak B(x\apr y,1)=\frak B(x,1)\frak B(y,1),\;\;\forall
x,y\in A,
\end{equation}
then the linear map $\vap:A\rightarrow\bfk$ defined by
\begin{equation}\notag
\vap(x):=\frak B(x,1),\;\;\forall x\in A,
\end{equation}
is an augmentation map.
\end{pro}
\begin{proof}
All the statements can be verified directly from the definitions. \end{proof}

\begin{ex}
Let $(A, \apr , {\bf 1},\vap)$ be an augmented unital commutative
$\bfk$-algebra. Let $\frak B$ be the bilinear form defined by
Eq.~(\mref{eq:a-B}). Suppose that $\frak B$ is nondegenerate. Then
$(A,\apr ,\frak B)$ is a symmetric Frobenius algebra. Let
$\phi^\sharp:A^\ast \rightarrow A$ be the linear isomorphism defined
by Eq.~(\mref{eq:phi}). Let $\{e_1={\bf 1},e_2,\cdots,e_n\}$ be a
basis of $A$ satisfying
$$\frak B(e_i,e_j)=\delta_{ij}, \;\;\forall i,j=1,\cdots, n.$$
Then $\phi\in A\otimes A$ is invariant and
$$\phi=\sum_{i=1}^n e_i\otimes e_i={\bf 1}\otimes {\bf
1}+\sum_{i=2}^n e_i\otimes e_i.$$ By Theorem~\ref{thm:nhybrb} and Corollary~\mref{cor:YBE-RB}, we
show that if $r$ satisfies Eqs.~(\mref{eq:nhybe}) and
(\mref{eq:phi-inv}), then the linear maps
 $P$ and $P'$ defined by Eq.~(\mref{eq:RB}) are
Rota-Baxter operators of weight $\lambda$.  Note that this conclusion also follows form Theorem~\mref{cor:Frobenius}, since in this case, $P=P_r$ and $P'=P_r^t$, where $P_r$ and $P_r^t$ are defined by Eq.~(\mref{eq:Pr}).
\end{ex}

\section{Classification of \inv solutions of \nybe
in low dimensions}
\mlabel{sec:exam}

In this section, we classify \inv solutions of the \nybe for $\mu \ne 0$ in the unital complex algebras
in dimensions two and three and find that all of them are obtained from
Rota-Baxter operators through Theorem~\mref{cor:Frobenius}.
It would be interesting to see what happens for algebras in higher dimensions.

\subsection{The classification in dimension two}\mlabel{ss:2dim}

The set of symmetric invariant tensors of a $\bfk$-algebra $A$ is a subspace of $A\ot A$ and is denoted by ${\rm Inv}(A)$.

There are two two-dimensional unital $\CC$-algebras whose
nonzero products with respect to a basis $\{e_1,e_2\}$ are given
by~\mcite{Pe}
\begin{eqnarray*}
&&(A1): e_1e_1=e_1, e_1e_2=e_2e_1=e_2;\\
&&(A2): e_1e_1=e_1, e_2e_2=e_2.
\end{eqnarray*}
By~\mcite{Liu}, for the algebra $(A1)$, there is only
one nonzero solution $r=\mu e_1\otimes e_1$ of the \nybe
Eq.~(\mref{eq:nhybe}). By Remark~\mref{rmk:none}, this solution is not \inv.

Consider the solutions of the \nybe in the algebra $(A2)$. Eight of the nine nonzero solutions are \inv, given by
\begin{eqnarray*}
&&r_1=\mu(e_1\otimes e_1+e_2\otimes e_2+e_1\otimes e_2),\;\;r_2=\mu(e_1\otimes e_1+e_2\otimes e_2+e_2\otimes
e_1),\\
&&r_3=\mu e_1\otimes e_2,\;\;r_4=\mu e_2\otimes e_1,\\
&&r_5=\mu (e_1\otimes e_1+e_1\otimes e_2),\;\;r_6=\mu (e_1\otimes e_1+e_2\otimes e_1),\\
&&r_7=\mu (e_2\otimes e_2+e_1\otimes e_2),\;\;r_8=\mu (e_2\otimes e_2+e_2\otimes e_1).
\end{eqnarray*}
Moreover, all of these solutions are obtained from
Rota-Baxter operators by Theorem~\mref{cor:Frobenius}.

To see this, note that $$r_2=\sigma (r_1),\;\; r_4=\sigma (r_3),\;\; r_6=\sigma
(r_5),\;\;r_8=\sigma (r_7),$$ and the unit of the  algebra $(A2)$
is $e_1+e_2$.  It is straightforward to show that ${\rm Inv}(A2)={\rm
span} \{e_1\otimes e_1, e_2\otimes e_2\}$. Let $\frak B_1$ and
$\frak B_2$ be the bilinear forms on $(A2)$ defined by
$$\frakB_1(e_1,e_1)=\frakB_1(e_2,e_2)=1,
\frakB_1(e_1,e_2)=\frakB_1(e_2,e_1)=0;$$$$
\frakB_2(e_1,e_1)=1,\frakB_2(e_2,e_2)=-1,
\frakB_2(e_1,e_2)=\frakB_2(e_2,e_1)=0.
$$
Then both $\frakB_1$ and $\frak B_2$ are symmetric, nondegenerate and invariant. Their corresponding symmetric, invariant tensors from Lemma~\mref{lem:Frobenius} are
$$\phi_1=e_1\otimes e_1+e_2\otimes e_2,\;\;
\phi_2=e_1\otimes e_1-e_2\otimes e_2,$$
so that $\frakB_i(x,y)=\langle
{\phi_i^\sharp}^{-1}(x),y\rangle$ for any $x,y\in (A2)$ and $i=1,2$.
Now the 8 \inv solutions of the \nybe satisfy
\begin{eqnarray*}
&&r_1+\sigma(r_1)=r_2+\sigma(r_2)=r_1+r_2=\mu\phi_1+\mu(e_1+e_2)\otimes
(e_1+e_2);\\
&&r_3+\sigma(r_3)=r_4+\sigma(r_4)=r_3+r_4=-\mu\phi_1+\mu(e_1+e_2)\otimes
(e_1+e_2);\\
&&r_5+\sigma(r_5)=r_6+\sigma(r_6)=r_5+r_6=\mu\phi_2+\mu(e_1+e_2)\otimes
(e_1+e_2);\\
&&r_7+\sigma(r_7)=r_8+\sigma(r_8)=r_7+r_8=-\mu\phi_2+\mu(e_1+e_2)\otimes
(e_1+e_2).
\end{eqnarray*}
By Theorem~\mref{cor:Frobenius}, their corresponding linear operators $P_{r_1}, P_{r_2}, P_{r_5}, P_{r_6}$ are
Rota-Baxter operators of weight $-\mu$ and $P_{r_3}, P_{r_4}, P_{r_7}, P_{r_8}$ are
Rota-Baxter operators of weight $\mu$.
\delete{
Explicitly, the operators are defined by
\begin{eqnarray*}&&P_{r_1}(e_1)=e_1+e_2, P_{r_1}(e_2)=e_2;\;\;P_{r_2}(e_1)=e_1,
P_{r_2}(e_2)=e_1+e_2;\\
&& P_{r_3}(e_1)=e_2, P_{r_3}(e_2)=0;\;\;P_{r_4}(e_1)=0,P_{r_4}(e_2)=e_1;\\
&& P_{r_5}(e_1)=e_1+e_2, P_{r_5}(e_2)=0;\;\;P_{r_6}(e_1)=e_1,P_{r_6}(e_2)=-e_1;\\
&& P_{r_7}(e_1)=e_2, P_{r_7}(e_2)=-e_2;\;\;P_{r_8}(e_1)=0,P_{r_8}(e_2)=-e_1-e_2.
\end{eqnarray*}
}
\subsection{The classification in dimension three}
\mlabel{ss:dim3}

Any three-dimensional unital $\mathbb C$-algebra is isomorphic to
one of the following five~\mcite{KSTT,St}, defined by their nonzero products on a basis $\{e_1,e_2,e_3\}$
\begin{eqnarray*}
&&(B1):e_1e_1=e_1,e_2e_2=e_2,e_3e_3=e_3;\\
&&(B2):e_1e_1=e_1,e_2e_2=e_2, e_3e_2=e_2e_3=e_3;\\
&&(B3):e_1e_1=e_1,e_1e_2=e_2e_1=e_2,e_1e_3=e_3e_1=e_3,e_2e_2=e_3;\\
&&(B4):e_1e_1=e_1,e_1e_2=e_2e_1=e_2,e_1e_3=e_3e_1=e_3,e_3e_2=e_2,e_3e_3=e_3;\\
&&(B5):e_1e_1=e_1,e_1e_2=e_2e_1=e_2,e_1e_3=e_3e_1=e_3.
\end{eqnarray*}

Solutions of the \nybe in these algebras were classified
in~\mcite{Liu}. For the algebras $(B3)$ and $(B5)$,
there is exactly one nonzero solution $r=\mu e_1\otimes e_1$ and it is not \inv.

For the algebra $(B4)$, it is straightforward to prove that
${\rm Inv}(B4)=0$. Then by~\mcite{Liu}, none of the nonzero solutions is \inv.

For the algebra $(B2)$, $e_1+e_2$ is the unit. Moreover, the
vector subspace $S$ spanned by $e_1$ and $e_2$ is a unital subalgebra of $(B2)$. It is in fact $(A2)$ in Section~\mref{ss:2dim}. As discussed there, there are 8 \inv solutions $r_i$, $1\leq i\leq 8$, of
the \nybe in $S$, together with
the corresponding Rota-Baxter operators $P_{r_i}, 1\leq i\leq 8$ on $(A2)$. In fact, they
are the only nonzero \inv solutions of Eq.~\meqref{eq:nhybe} in $(B2)$. The corresponding Rota-Baxter operators on $(B2)$ are derived from $P_{r_i}, i=1, \cdots, 8$ by setting $P_{r_i}(e_3)=0$, as shown in~\mcite{AB}.

For the algebra $(B1)$,  among
the total of 73 nonzero solutions of the \nybe given in~\mcite{Liu},  there are exactly 48 nonzero solutions that are \inv. All of these solutions are obtained from Rota-Baxter
operators thanks to Theorem~\mref{cor:Frobenius}.

Note that the unit {\bf 1} is $e_1+e_2+e_3$ and
$${\rm Inv}(B1)={\rm
    span}\{e_1\otimes e_1,e_2\otimes e_2,e_3\otimes e_3\}.$$
Set
\begin{eqnarray*}
    &&\phi_1:=e_1\otimes e_1+e_2\otimes e_2+e_3\otimes e_3,\;\;\phi_2:=e_1\otimes e_1+e_2\otimes e_2-e_3\otimes e_3,\\
    &&\phi_3:=e_1\otimes e_1-e_2\otimes e_2+e_3\otimes
    e_3,\;\;\phi_4:=-e_1\otimes e_1+e_2\otimes e_2+e_3\otimes e_3.
\end{eqnarray*}

According to their \tsymms
$$\sr:=r+\sigma(r)-\mu (\bfone\ot \bfone),$$
these 48 solutions are grouped together as follows.
\begin{eqnarray*}
    r_1&=& \mu(e_2\otimes e_1+e_3\otimes e_1+e_3\otimes
    e_2),\;\;
    r_2= \mu(e_1\otimes e_2+e_1\otimes e_3+e_2\otimes
    e_3),\\
    r_3&=& \mu(e_2\otimes e_1+e_2\otimes e_3+e_3\otimes
    e_1),\;\;
    r_4= \mu(e_1\otimes e_2+e_3\otimes e_2+e_1\otimes
    e_3),\\
    r_5&=& \mu(e_1\otimes e_3+e_2\otimes e_1+e_2\otimes
    e_3),\;\;
    r_6= \mu(e_3\otimes e_1+e_1\otimes e_2+e_3\otimes
    e_2),
\end{eqnarray*}
for which $\sr=-\mu \phi_1$ and hence their corresponding linear operators in Theorem~\mref{cor:Frobenius} are Rota-Baxter operators of weight $\mu$. Furthermore,
$$r_{i}= r_{i-6}+\mu\phi_1, \quad  7\leq i\leq 12,
$$
for which $\sr=\mu \phi_1$ and hence correspond to Rota-Baxter operators of weight $-\mu$.
$$ r_i=r_{i-12}+\mu(e_3\ot e_3), \quad 13\leq i\leq 18,
$$
for which $\sr=-\mu \phi_2$ and hence correspond to Rota-Baxter operators of weight $\mu$.
$$ r_i=r_{r-18}+\mu(e_1\otimes e_1+e_2\otimes e_2), \quad 19\leq i\leq 24,
$$
for which $\sr=\mu \phi_2$ and hence correspond to Rota-Baxter operators of weight $-\mu$.
$$ r_i=r_{i-24}+\mu(e_2\otimes e_2), \quad 25\leq i\leq 30,
$$
for which $\sr=-\mu \phi_3$ and hence correspond to Rota-Baxter operators of weight $\mu$.
$$ r_i=r_{i-30}+\mu(e_1\otimes e_1+e_3\otimes e_3), \quad 31\leq i\leq 36,
$$
for which $\sr=\mu \phi_3$ and hence correspond to Rota-Baxter operators of weight $-\mu$.
$$ r_i=r_{i-36}+\mu(e_1\otimes e_1), \quad 36\leq i\leq 42,
$$
for which $\sr=-\mu \phi_4$ and hence correspond to Rota-Baxter operators of weight $\mu$.
$$ r_i=r_{i-42}+\mu(e_2\otimes e_2+e_3\otimes e_3), \quad 43\leq i\leq 48,$$
for which $\sr=\mu \phi_4$ and hence correspond to Rota-Baxter operators of weight $-\mu$.

\smallskip

 \noindent
 {\bf Acknowledgements.}  This work is supported by
 National Natural Science Foundation of China (Grant Nos. 11931009,  11771190).  C. Bai is also
supported by the Fundamental Research Funds for the Central
Universities and Nankai ZhiDe Foundation. Y. Zhang is supported by
China Scholarship Council to visit  University of Southern California and he thanks Professor Susan Montgomery for hospitality during his visit.

\end{document}